\documentclass{amsart}

\usepackage{epsfig}
\usepackage{amsmath}
\usepackage{amssymb,mathrsfs}

\usepackage[all]{xy}

\usepackage{color}
\usepackage{tabularx}

\def\comment#1{{\color{blue}[{\sf #1}]}}

\def\Z{{\mathbb Z}}
\def\Q{{\mathbb Q}}

\def\R{{\mathbb R}}
\def\C{{\mathbb C}}
\def\P{{\mathbb P}}

\def\D{{\mathbb D}}

\def\L{{\mathbb L}}

\def\kk{{\Bbbk}}		

\def\cG{{\mathcal G}}
\def\M{{\mathcal M}}
\def\cC{{\mathcal C}}
\def\cD{{\mathcal D}}

\def\cO{{\mathcal O}}

\def\U{{\mathcal U}}

\def\sB{{\mathscr B}}
\def\sD{{\mathscr D}}
\def\sE{{\mathscr E}}
\def\sF{{\mathscr F}}
\def\sL{{\mathscr L}}
\def\sP{{\mathscr P}}
\def\sV{{\mathscr V}}

\def\g{{\mathfrak g}}

\def\p{{\mathfrak p}}
\def\sp{{\mathfrak{sp}}}

\def\u{{\mathfrak u}}

\def\e{{\epsilon}}
\def\G{{\Gamma}}

\def\bI{\boldsymbol{I}}
\def\bL{\boldsymbol{L}} 
\def\bP{\boldsymbol{P}}
\def\bV{\boldsymbol{V}}
\def\bW{\boldsymbol{W}}

\def\bLambda{\boldsymbol{\Lambda}}

\def\alphahat{\hat{\alpha}}
\def\gammahat{\hat{\gamma}}
\def\kappahat{\hat{\kappa}}

\def\alphatilde{{\tilde{\alpha}}}
\def\gammatilde{{\tilde{\gamma}}}
\def\chitilde{{\tilde{\chi}}}
\def\rhotilde{{\tilde{\rho}}}
\def\phitilde{{\tilde{\varphi}}}

\def\Shat{{\widehat{S}}}

\def\Zhat{{\widehat{\Z}}}

\def\xbar{{\overline{x}}}
\def\Thetabar{{\overline{\Theta}}}

\def\Fbar{\overline{F}}
\def\Kbar{{\overline{K}}}

\def\Xbar{{\overline{X}}}

\def\sVbar{\overline{\sV}}

\def\Pdual{\check{\bP}}
\def\Ldual{\check{\bL}}

\def\Hdual{\check{H}}

\def\vv{{\vec{\mathsf v}}}

\def\MHS{{\mathsf{MHS}}}

\def\Vec{{\mathsf{Vec}}}
\def\Rep{{\mathsf{Rep}}}

\def\cts{{\mathrm{cts}}}
\def\un{{\mathrm{un}}}
\def\ss{{\mathrm{ss}}}
\def\dR{{\mathrm{dR}}}
\def\prol{{\text{pro-}\ell}}

\def\To{\longrightarrow}
\def\bdot{{\bullet}}

\def\blank{\phantom{x}}

\def\PD{{\mathcal{D}}}
\def\KV{{\mathrm{KVI}}}

\def\cs{\beta_{C\!S}}
\def\csp{\beta_{K\!K}}
\def\csdual{{\check{\beta}_{C\!S}}}
\def\cspdual{\check{\beta}_{K\!K}}

\def\bil{{\langle\blank,\blank\rangle}}
\def\gold{\{\blank,\blank\}}

\def\ll{\langle\!\langle}
\def\rr{\rangle\!\rangle}

\def\unit{{\mathbf{1}}}

\def\MT{{\mathrm{MT}}}
\def\Gm{{\mathbb{G}_m}}
\def\Sp{{\mathrm{Sp}}}

\def\Ch{{\mathrm{Ch}}}

\def\BlR#1{\smash{\widehat{\Bl}}^\R_{#1}\,}

\newcommand\im{\operatorname{im}}

\newcommand\Hom{\operatorname{Hom}}
\newcommand\Ext{\operatorname{Ext}}
\newcommand\Isom{\operatorname{Isom}}

\newcommand\End{\operatorname{End}}
\newcommand\Aut{\operatorname{Aut}}
\newcommand\Diff{\operatorname{Diff}}

\newcommand\Der{\operatorname{Der}}
\newcommand\Gr{\operatorname{Gr}}

\newcommand\cone{\operatorname{cone}}
\newcommand\Bl{\operatorname{Bl}}

\newcommand\Sym{\operatorname{Sym}}
\newcommand\Gal{\operatorname{Gal}}

\numberwithin{equation}{section}

\newtheorem{theorem}{Theorem}[section]
\newtheorem{lemma}[theorem]{Lemma}
\newtheorem{proposition}[theorem]{Proposition}
\newtheorem{corollary}[theorem]{Corollary}
\newtheorem{bigtheorem}{Theorem}
\newtheorem{bigcorollary}[bigtheorem]{Corollary}

\theoremstyle{definition}
\newtheorem{definition}[theorem]{Definition}

\theoremstyle{remark}
\newtheorem{remark}[theorem]{Remark}

\begin{document}

\title{Hodge Theory of the Goldman Bracket}

\dedicatory{Dedicated to the memory of Stefan Papadima.}

\author{Richard Hain}
\address{Department of Mathematics\\ Duke University\\
Durham, NC 27708-0320}
\email{hain@math.duke.edu}

\thanks{Supported in part by grant DMS-1406420 from the National Science Foundation.}
\thanks{ORCID iD: {\sf orcid.org/0000-0002-7009-6971}}

\date{\today}



\begin{abstract}
In this paper we show that, after completing in the $I$-adic topology, the Goldman bracket on the space spanned by the closed geodesics on a smooth, complex algebraic curve $X$ is a morphism of mixed Hodge structure. We prove a similar statements for the natural action of the loops in X on paths from one boundary vector to another.
\end{abstract}

\maketitle


\section{Introduction}

Denote the set of free homotopy classes of maps $S^1 \to X$ in a topological space $X$ by $\lambda(X)$ and the free $R$-module it generates by $R\lambda(X)$. When $X$ is an oriented surface, Goldman \cite{goldman} defined a binary operation on $\Z\lambda(X)$, giving it the structure of a Lie algebra over $\Z$. Briefly, the bracket of two oriented loops $A$ and $B$ is defined by choosing transverse, immersed representatives $\alpha$ and $\beta$ of the loops, and then defining
$$
\{A,B\} = \sum_p \e_p(\alpha,\beta)\, \alpha \#_p \beta
$$
where the sum is taken over the points $p$ of intersection of $\alpha$ and $\beta$, $\e_p(\alpha,\beta)\in \{\pm 1\}$ denotes the local intersection number of $\alpha$ and $\beta$ at $p$, and where $\alpha \#_p \beta$ denotes the homotopy class of the oriented loop obtained by joining $\alpha$ and $\beta$ at $p$ by a simple surgery.

The powers of the augmentation ideal $I$ of $R\pi_1(X,x)$ define a topology on $R\pi_1(X,x)$. It induces a topology on $R\lambda(X)$. Denote their $I$-adic completions by $R\pi_1(X,x)^\wedge$ and $R\lambda(X)^\wedge$. Kawazumi and Kuno \cite{kk:groupoid} showed that the Goldman bracket is continuous in the $I$-adic topology and thus induces a bracket on $\Q\lambda(X)^\wedge$.

If $X$ is the complement of a (possibly empty) finite subset of a compact Riemann surface and $x\in X$, then there is a canonical (pro) mixed Hodge structure \cite{hain:dht,morgan} on $\Q\pi_1(X,x)^\wedge$. It induces a pro-mixed Hodge structure on $\Q\lambda(X)^\wedge$ via the quotient map $\Q\pi_1(X,x)^\wedge \to \Q\lambda(X)^\wedge$. This mixed Hodge structure is independent of the choice of the base point $x$. In order that the Goldman bracket preserve the weight filtration of $\Q\lambda(X)^\wedge$, we have to shift the weight filtration by $2$. We do this by tensoring it with $\Q(-1)$, the 1-dimensional Hodge structure of weight $2$ and type $(1,1)$.

\begin{bigtheorem}
\label{thm:free}
The Goldman bracket on $\Q\lambda(X)^\wedge\otimes\Q(-1)$ is a morphism of pro-mixed Hodge structures. Equivalently, the Goldman bracket
$$
\{\blank,\blank\} :
\Q\lambda(X)^\wedge \otimes \Q\lambda(X)^\wedge
\to \Q\lambda(X)^\wedge\otimes \Q(1)
$$
is a morphism of pro-mixed Hodge structures.
\end{bigtheorem}


Goldman defined his bracket to describe the Poisson bracket of certain functions on the symplectic manifold $\Hom(\pi_1(X,x),G)/G$, where $G$ is a Lie group. Our interest, as explained below, lies in its likely role in studying motives constructed from smooth curves and their moduli spaces. We are also interested in related applications to the Kashiwara--Vergne problem as described in the work of Alekseev, Kawazumi, Kuno and Naef \cite{akkn,akkn2}.

For those applications, one needs a more general setup. Denote the torsor of homotopy classes of paths in $X$ from $x_0$ to $x_1$ by $\pi(X;x_0,x_1)$. When $X$ is an oriented surface with boundary $\partial X$, and when $x_0,x_1 \in \partial X$, there is an action
$$
\kappa : \Q\lambda(X) \otimes \Q\pi(X;x_0,x_1) \to \Q\pi(X;x_0,x_1)
$$
which was defined by Kawazumi and Kuno \cite{kk:log} and whose definition is similar to that of the Goldman bracket. They showed that $\kappa$ is continuous in the $I$-adic topology and thus induces a map (the ``KK-action'')
$$
\kappa : \Q\lambda(X)^\wedge \otimes \Q\pi(X;x_0,x_1)^\wedge \to \Q\pi(X;x_0,x_1)^\wedge.
$$

Complex algebraic curves do not have boundary curves. Because of this, we have to replace the boundary points $x_0$ and $x_1$ by tangent vectors at the cusps. More precisely, suppose that $\Xbar$ is a compact Riemann surface and that $X$ is the complement $\Xbar-D$ of a finite subset of $\Xbar$. Suppose that $p_0,p_1\in D$ (not necessarily distinct) and that $\vv_j$ is a non-zero tangent vector in $T_{p_j}\Xbar$. Then one has has the path torsor $\pi(X;\vv_0,\vv_1)$ of paths in $X$ from $\vv_0$ to $\vv_1$. (The definition, due to Deligne \cite[\S15]{deligne:p1}, is recalled in Section~\ref{sec:tangent_vecs}.) Results from \cite{hain:dht2,hain-zucker} imply that the $I$-adic completion $\Q\pi(X;\vv_0,\vv_1)^\wedge$ has a natural (limit) mixed Hodge structure.

\begin{bigtheorem}
\label{thm:action}
With this notation, the KK-action
$$
\kappa : \Q\lambda(X)^\wedge \otimes \Q\pi(X;\vv_0,\vv_1)^\wedge \to \Q\pi(X;\vv_0,\vv_1)^\wedge\otimes \Q(1)
$$
is a morphism of pro-MHS.
\end{bigtheorem}

When $\vv_0=\vv_1$, we obtain an action of $\Q\lambda(X)^\wedge\otimes\Q(-1)$ on $\Q\pi_1(X,\vv)^\wedge$.

\begin{bigcorollary}
The Lie algebra homomorphism
$$
\kappahat : \Q\lambda(X)^\wedge\otimes \Q(-1) \to \Der \Q\pi_1(X,\vv)^\wedge
$$
induced by the Kawazumi--Kuno action is a morphism of pro-MHS.
\end{bigcorollary}

Theorem~\ref{thm:action} implies Theorem~\ref{thm:free}. While it is similar to Theorem~\ref{thm:free}, it is much harder to prove. The first step in its proof is to establish the following weaker version.

\begin{bigtheorem}
\label{thm:naive}
If $X$ is the complement of a finite set in a compact Riemann surface and $x_0,x_1 \in X$, then the natural action
$$
\kappa : \Q\lambda(X-\{x_0,x_1\})^\wedge\otimes \Q\pi(X;x_0,x_1)^\wedge \to \Q\pi(X;x_0,x_1)^\wedge\otimes\Q(1)
$$
is a morphism of pro-MHS.
\end{bigtheorem}

The proof of Theorem~\ref{thm:action} is completed by showing that the surjection
$$
\Q\lambda(X-\{x_0,x_1\})^\wedge \to \Q\lambda(X)^\wedge
$$
induced by the inclusion (a morphism of pro-MHS), splits in the category of pro-MHS after one ``takes the limit as $x_j \to p_j$ in the direction of $\vv_j$.'' This involves technical Hodge theory.

We also show that the mixed Hodge theory of the KK-action is compatible with the mixed Hodge structures on completions of mapping class groups constructed in \cite{hain:torelli}. Suppose that $\Xbar$ is a compact oriented surface and that $V=\{\vv_a : a\in A\}$ is a finite collection of non-zero tangent vectors indexed by a finite subset $A$ of $\Xbar$. Set $X=\Xbar-A$. For each $a\in A$, the mapping class group $\G_{\Xbar,V}$ acts on $\pi_1(X,\vv_a)$. This induces an action of the completion $\cG_{\Xbar,V}$ of $\G_{\Xbar,V}$ on $\Q\pi_1(X,\vv_a)^\wedge$ and a Lie algebra homomorphism
\begin{equation}
\label{eqn:univ}
d\rho_a : \g_{\Xbar,V} \to \Der \Q\pi_1(X,\vv_a)^\wedge.
\end{equation}
Results of Kawazumi and Kuno \cite{kk:groupoid} imply that there is a canonical Lie algebra homomorphism
$$
\varphi : \g_{\Xbar,V} \to \lambda(X)^\wedge
$$
which lifts (\ref{eqn:univ}) in the sense that the composite
$$
\g_{\Xbar,V} \to \lambda(X)^\wedge \to \Der \Q\pi_1(X,\vv_a)^\wedge
$$
is (\ref{eqn:univ}). (Details can be found in Section~\ref{sec:lie}.) When $\Xbar$ is a compact Riemann surface, $\g_{\Xbar,V}$ has a canonical MHS. Using results of \cite{hain:torelli}, we show that this lift is a morphism of MHS.

\begin{bigtheorem}
\label{thm:mcg}
 If $\Xbar$ has the structure of a Riemann surface, then the Lie algebra homomorphism
$$
\phitilde : \g_{\Xbar,V} \to \Q\lambda(X)^\wedge\otimes \Q(-1)
$$
is a morphism of pro-MHS.
\end{bigtheorem}

These results are proved by factoring the Goldman bracket and the KK-action as a composite of maps each of whose continuous duals can be expressed in terms of differential forms. These factorizations appear in the work \cite[\S3]{kk:log} of Kawazumi and Kuno, but have antecedents in the work of Chas and Sullivan \cite{chas-sullivan} on string topology. We express several of the maps in the factorization in terms of Chen's iterated integrals. These formulas allow us to prove that, with the appropriate hypotheses, these maps preserve the Hodge and weight filtrations defined in \cite{hain:dht} and are thus morphisms of mixed Hodge structure.

\subsection{Splittings}

One motivation for this work was to use Hodge theory to find and study compatible splittings of the various objects (such as $\Q\lambda(X)^\wedge$, $\Q\pi_1(X,\vv)^\wedge$, $\g_{\Xbar,V}$, etc.) associated to a compact Riemann surface $\Xbar$ with a finite collection $V$ of non-zero tangent vectors. As explained in \cite{akkn2}, finding such compatible splittings is related to constructing solutions of the Kashiwara--Vergne problem. The category $\MHS$ of graded polarizable $\Q$-MHS is a $\Q$-linear tannakian category, and therefore the category of representations of an affine $\Q$ group, which we denote by $\pi_1(\MHS)$.\footnote{The fiber functor is the underlying $\Q$ vector space.} Its reductive quotient is $\pi_1(\MHS^\ss)$, the tannakian fundamental group of the category of semi-simple MHS. As we explain in Section~\ref{sec:tannaka}, there is a canonical cocharacter $\chi : \Gm \to \pi_1(\MHS^\ss)$. Each of its lifts $\chitilde : \Gm \to \pi_1(\MHS)$ gives a splitting
$$
V_\Q \cong \bigoplus_{m\in \Z} \Gr^W_m V_\Q
$$
of the weight filtration of the underlying $\Q$ vector space of each graded polarizable MHS. These isomorphisms are natural in that they are respected by morphisms of MHS, but are not canonical as they depend on the choice of the lift $\chitilde$.

The {\em Mumford--Tate} group $\MT_V$ of a graded polarizable MHS $V$ is defined to be the image of the homomorphism
$$
\pi_1(\MHS) \to \Aut(V).
$$
The Hodge splittings form a principal homogeneous space under the action of the unipotent radical $U_V$ of $\MT_V$.

As above, $\Xbar$ is a compact oriented surface and $V=\{\vv_a : a\in A\}$ is a finite collection of non-zero tangent vectors indexed by a finite subset $A$ of $\Xbar$. Set $X=\Xbar-A$. The Lie algebra $\p(X,\vv_a)$ of the unipotent completion of $\pi_1(X,\vv_a)$ is the set of primitive elements of the the completed group algebra $\Q\pi_1(X,\vv_a)^\wedge$. Denote the Mumford--Tate group of $\p(X,\vv_o)$ by $\MT_{X,\vv_o}$ and its prounipotent radical by $U_{X,\vv_o}^\MT$. This group depends non-trivially on $(X,\vv_o)$.

\begin{bigtheorem}
\label{thm:splittings}
Each choice of a lift $\chitilde : \Gm \to \pi_1(\MHS)$ of the canonical central cocharacter $\chi : \Gm \to \pi_1(\MHS^\ss)$ determines compatible splittings of the weight filtrations on each of
$$
\p(X,\vv_a),\quad \Q\pi_1(X,\vv_a)^\wedge,\quad \g_{\Xbar,V},\quad \Q\lambda(X)^\wedge\otimes \Q(-1),\quad \Der \Q\pi_1(X,\vv_a)^\wedge.
$$
Under all such splittings, the algebraic operations, such as Goldman bracket, are graded. Each choice of $\chitilde$ determines a ``symplectic Magnus expansion'' of $\pi_1(X,\vv_a)$. The splittings constructed from Hodge theory are a torsor under $U_{X,\vv_o}^\MT$.
\end{bigtheorem}

The isomorphism of the completed Goldman Lie algebra with the ``degree completion'' of its associated graded Lie algebra was originally proved by Kawazumi and Kuno in \cite{kk:groupoid} for surfaces with one boundary component. The general case was proved by Alekseev, Kawazumi, Kuno and Naef in \cite{akkn,akkn2}.

In general, the splitting of $\p(X,\vv_a)$ determined by $\chitilde$ gives a solution of the first Kashiwara--Vergne problem $\mathrm{KVI}^{(g,|A|)}$ as defined in \cite{akkn1}. Complete solutions of the Kashiwara--Vergne problem correspond, by \cite{akkn2}, to splittings of the weight filtration of $\Q\lambda(X)$ that are also compatible with the completed Turaev cobracket \cite{turaev:78,turaev:skein,kk:intersections}
$$
\delta_\xi : \Q\lambda(X)^\wedge \to \Q\lambda(X)^\wedge\otimes \Q\lambda(X)^\wedge
$$
corresponding to a framing $\xi$ of $X$. This is achieved in \cite{hain:turaev} using Hodge theory.

\subsection{Are the Goldman bracket and KK-action motivic?}

The results above suggest that the Goldman bracket and the KK-action are motivic. In particular, they suggest that if $X$ is defined over a subfield $K$ of $\C$, then the profinite completion
$$
\widehat{\Z\lambda(X_{/\Kbar})} \otimes \widehat{\Z\lambda(X_{/\Kbar})} \to \widehat{\Z\lambda(X_{/\Kbar})}\otimes \Zhat(1)
$$
of the Goldman bracket is $\Gal(\Kbar/K)$-invariant. Here ${\lambda}(X_{/\Kbar})$ denotes the set of conjugacy classes in geometric \'etale fundamental group of $X$ and $\Zhat$ denotes the profinite completion of the integers. The analogous statement should also hold for the KK-action. In the case of prounipotent completions used in this paper, we expect to be able to prove that in the prounipotent case, after extending scalars to $\Q_\ell$, that the Goldman bracket
$$
\Q_\ell \lambda(X_{\Kbar})^\prol \otimes \Q_\ell \lambda(X_{\Kbar})^\prol \to \Q_\ell \lambda(X_{\Kbar})^\prol\otimes\Q_\ell(1)
$$
is $\Gal(\Kbar/K)$-equivariant, and similarly for the KK-action.

\subsection{Outline}

The main results are proved using known factorizations of the Goldman bracket and the KK-action. These are recalled in Sections~\ref{sec:path_space} and \ref{sec:goldman}.  The strategy is to prove that each map in the factorization is a morphism of mixed Hodge structure. The first step is to show that each can be expressed in terms of differential forms. Those maps that are not Poincar\'e duality are expressed in terms of iterated integrals. This is done in Sections~\ref{sec:it_ints} to \ref{sec:kk_DR}. The proofs are completed using these de~Rham descriptions to show that all maps in the factorizations of the bracket and action are morphisms of mixed Hodge structure. Below is a section by section description of the contents of the paper.

Since we use the de~Rham theory of path spaces to prove the results, we begin, in Section~\ref{sec:path_space}, with a review of path spaces and the associated fibrations and local systems. In Section~\ref{sec:goldman} we recall the homological definitions of the Goldman bracket and KK-action. In preparation for stating the path space de~Rham theorems needed later in the paper, we review, in Section~\ref{sec:topology}, several basic facts from rational homotopy theory, including unipotent completion and the definition of rational $K(\pi,1)$ spaces. Section~\ref{sec:it_ints} is a review of Chen's iterated integrals, the cyclic bar construction and several of the de~Rham theorems for path spaces. We also prove a new (though not unexpected) de~Rham theorem for the degree 0 de~Rham cohomology of the free loop space of a non-simply connected manifold. This is in preparation for giving the de~Rham description of the continuous dual of the Goldman bracket, which we derive in Section~\ref{sec:goldman_DR}. This is followed with a de~Rham description of the continuous dual of the KK-action in Section~\ref{sec:kk_DR}.

With the de~Rham description of the duals of the Goldman bracket and the KK-action established, we embark on the Hodge theoretic aspects in Section~\ref{sec:Hodge}. This begins with a very brief review of the basics of mixed Hodge structures (e.g., definition, exactness properties). Because of the relevance of splittings of weight filtrations to the Kashiwara--Vergne problem \cite{akkn1}, there is a brief review of Tannaka theory, which is used to explain the existence of natural splittings of the weight filtrations of mixed Hodge structures and to define Mumford--Tate groups of mixed Hodge structures. Admissible variations of MHS are reviewed and the compatibility of Poincar\'e duality with Hodge theory with coefficients in an admissible variation is established using Saito's fundamental work \cite{saito:mhm}. Results from \cite{hain-zucker} are recalled, which imply that certain local systems constructed from iterated integrals are admissible variations of MHS. The section concludes with a brief review of limit mixed Hodge structures.

The de~Rham descriptions of the Goldman bracket and the KK-action are combined with results from Hodge theory to prove Theorems~\ref{thm:free} and \ref{thm:naive} in Section~\ref{sec:proofs}. The proof of Theorem~\ref{thm:action} is considerably more difficult and is proved in Section~\ref{sec:proof_technical} using a splitting result for limit mixed Hodge structures. Theorems~\ref{thm:mcg} and \ref{thm:splittings} are proved in Section~\ref{sec:mcg}.

\bigskip

\noindent{\em Acknowledgments:} I would like to thank Anton Alekseev for stimulating my interest in the Kashiwara--Vergne problem. That and his work \cite{akkn1,akkn,akkn2} with Nariya Kawazumi, Yusuke Kuno and Florian Naef aroused my interest in trying to prove that the Goldman bracket was compatible with Hodge theory. I would like to thank all of them for their patient and helpful explanations, especially Nariya Kawazumi who showed me the cohomological factorization of the Goldman bracket. I am also grateful to the referees for their careful reading of the manuscript.

\section{Notation and Conventions}

Suppose that $X$ is a topological space. There are two conventions for multiplying paths. We use the topologist's convention: The product $\alpha\beta$ of two paths $\alpha,\beta : [0,1]\to X$ is defined when $\alpha(1) = \beta(0)$. The product path traverses $\alpha$ first, then $\beta$. We will denote the set of homotopy classes of paths from $x$ to $y$ in $X$ by $\pi(X;x,y)$. In particular, $\pi_1(X,x) = \pi(X;x,x)$. The fundamental groupoid of $X$ is the category whose objects are $x\in X$ and where $\Hom(x,y) = \pi(X;x,y)$.

A {\em local system} $\bV$ over $X$ of $R$-modules ($R$ a ring) is a functor from the fundamental groupoid of $X$ to the category of $R$-modules. The fiber of the the local system over $x\in X$ (i.e., value at $x$) will be by $V_x$. When $X$ is locally contractible, there is an equivalence between the category of local systems of $R$-modules and the category of locally constant sheaves of $R$-modules. Sometimes we will regard a local system as a locally constant sheaf, and vice-versa.

When $X$ is compact, $Y$ a closed subspace, and $\sF$ a sheaf on $X-Y$, we define the relative cohomology group $H^\bdot(X,Y;\sF)$ by
$$
H^\bdot(X,Y;\sF) := H^\bdot(X; i_!\sF),
$$
where $i : X-Y \hookrightarrow X$ denotes the inclusion.

If $\bV$ is a local system of finite dimensional $\kk$-vector spaces over $X$, and if $X$ and $Y$ are finite complexes, define
$$
H_\bdot(X,Y;\bV) = \Hom_\kk(H^\bdot(X,Y;\bV^\ast),\kk)
$$
where $\bV^\ast$ denotes the dual local system.

Complex algebraic varieties will be considered as complex {\em analytic} varieties. In particular, $\cO_X$ will denote the sheaf of holomorphic functions on $X$ in the complex topology.

For clarity, we have attempted to denote complex algebraic and analytic varieties by the roman letter $X$, $Y$, etc and arbitrary smooth manifolds (and differentiable spaces) by the letters $M$, $N$, etc. This is not always possible. The complex of smooth $\kk$-valued differential forms on a smooth manifold $M$ will be denoted by $E^\bdot_\kk(M)$.

\section{Path Spaces and Local Systems}
\label{sec:path_space}

Throughout this section, $M$ is a smooth manifold and $\kk$ is a commutative ring.

\subsection{Path spaces}

The {\em path space} of $M$ is the set
$$
PM = \{\gamma : [0,1] \to M : \gamma \text{ is piecewise smooth}\}
$$
endowed with the compact-open topology. For each $t\in [0,1]$ we have the evaluation map
$$
p_t : PM \to M,\qquad \gamma \mapsto \gamma(t).
$$
The map
\begin{equation}
\label{eqn:PM}
p_0 \times p_1 : PM \to M\times M
\end{equation}
is a fibration (i.e., a {\em Hurewicz fibration}). Its fiber over $(x_0,x_1)$, the space of paths in $M$ from $x_0$ to $x_1$, will be denoted by $P_{x_0,x_1}M$.

The total space of the restriction
$$
p : \Lambda M \to M
$$
of (\ref{eqn:PM}) to the diagonal $M \subset M\times M$ is the free loops space $\{\alpha : S^1 \to M\}$ of $M$; the projection $p$ takes $\alpha : S^1 \to M$ to $\alpha(0)$, where we view $S^1$ as $[0,1]/(0\sim 1)$. The space of loops in $M$ based at $x\in M$, denoted $\Lambda_x M$, is $p^{-1}(x)$. Note that $P_{x,x}M = \Lambda_x M$ for all $x\in M$.

\begin{proposition}
We have
$$
\lambda(M) = \pi_0(\Lambda M),\quad \pi(M;x_0,x_1) = \pi_0(P_{x_0,x_1}M)\text{ and } \pi_1(M,x) = \pi_0(\Lambda_x M)
$$
and
$$
H_0(\Lambda M;\kk) = \kk \lambda(M) \text{ and } H_0(P_{x_0,x_1}M;\kk) = \kk\pi(M;x_0,x_1).
$$
\end{proposition}

\subsection{Local systems}

Two local systems play an important role in the story.

\begin{definition}
\label{def:L}
Define $\bL_M$ to be the local system of $\kk$-modules over $M$ whose fiber $\bL_{M,x}$ over $x\in M$ is $H_0(\Lambda_x M;\kk)$. 
\end{definition}

The parallel transport action
$$
\pi(M;x,y) \to  \Isom_\kk\big(H_0(\bL_x M;\kk), H_0(\bL_y M;\kk)\big)
$$
is given by conjugation.

For each pair $(x_0,x_1) \in M^2$, the evaluation map
$$
p : P_{x_0,x_1} M \to M,\qquad \gamma \mapsto \gamma(1/2)
$$
is a fibration, where $p = p_{\frac{1}{2}}$. Denote its fiber over $x \in M$ by $P_{x_0,x,x_1} M$. Observe that path multiplication induces a homeomorphism
$$
P_{x_0,x}M \times P_{x,x_1} M \to P_{x_0,x,x_1}M
$$
as each element $\gamma$ of $P_{x_0,x,x_1}M$ can be uniquely factored $\gamma'\gamma''$, where $\gamma' \in P_{x_0,x} M$ and $\gamma'' \in P_{x,x_1}M$.

\begin{definition}
\label{def:P}
Define $\bP_{M;x_0,x_1}$ to be the local system of $\kk$-modules over $M$ whose fiber $\bP_{M;x_0,x,x_1}$ over $x\in M$ is $H_0(P_{x_0,x,x_1}M;\kk)$. 
\end{definition}

The parallel transport function
$$
\pi(M;x,y) \to
\Isom_\kk\big(H_0(P_{x_0,x,x_1}M;\kk), H_0(P_{x_0,y,x_1}M;\kk)\big)
$$
is induced by the map
\begin{align*}
P_{x,y}M \times \big(P_{x_0,x}M \times P_{x,x_1}M\big)
&\to P_{x_0,y}M \times P_{y,x_1} M
\cr
(\mu,\gamma',\gamma'')\qquad &\mapsto \quad (\gamma'\mu,\mu^{-1}\gamma'').
\end{align*}

\begin{proposition}
\label{prop:pairing}
For each subset $U$ of $M-\{x_0,x_1\}$, there are well-defined pairings
\begin{equation}
\label{eqn:pairings}
\mu : \bL_U \otimes j^\ast\bP_{M;x_0,x_1} \to j^\ast\bP_{M;x_0,x_1}
\text{ and } \mu : \bL_U \otimes \bL_U \to \bL_U,
\end{equation}
where $j : U \hookrightarrow M$ denotes the inclusion.
\end{proposition}

\begin{proof}
On the fiber over $x\in U$, the first pairing is induced by the map
$$
\pi_1(U,x) \times \big(\pi(M;x_0,x) \times \pi(M;x,x_1)\big)
\to \pi(M;x_0,x)\times\pi(M;x,x_1)
$$
that takes $(\alpha,\gamma',\gamma'')$ to $(\gamma'\alpha,\gamma'')$. The second pairing is induced by the multiplication map $\Lambda_x U \times \Lambda_x U \to \Lambda_x U$.
\end{proof}

\begin{proposition}
\label{prop:k_pi_1}
If $M$ is connected, then there are natural isomorphisms
$$
H_0(M;\bL_M) \cong H_0(\Lambda M)
\text{ and }
H_0(M;\bP_{M;x_0,x_1}) \cong H_0(P_{x_0,x_1}M).
$$
If $M$ is a $K(\pi,1)$, then there are natural isomorphisms
$$
H_j(M;\bL_M) \cong H_j(\Lambda M)
\text{ and }
H_j(M;\bP_{M;x_0,x_1}) \cong H_j(P_{x_0,x_1}M).
$$
\end{proposition}

\begin{proof}
The assertions for $\Lambda M$ are proved by applying the Serre spectral sequence to the fibration $\Lambda M \to M$; the assertions for $P_{x_0,x_1}M$ are obtained by applying the Serre spectral sequence to the fibration $p_{1/2} : P_{x_0,x_1} M\to M$. This give the result in degree 0. If $M$ is a $K(\pi,1)$, then for all $x,y\in M$, $P_{x,y}M$ is a disjoint union of contractible spaces, which implies that the spectral sequence collapses at $E^2$, which proves the remaining assertions.
\end{proof}

\subsection{Chas--Sullivan maps}
\label{sec:CS}

Each $\alpha \in \Lambda M$ gives rise to a section $\alphahat$ of $\alpha^\ast\bL_M$. It is induced by the lift $\alphatilde$
$$
\xymatrix{
& \Lambda M \ar[d]^p \cr
[0,1] \ar[ur]^{\alphatilde} \ar[r]_\alpha & M.
}
$$
of $\alpha$ defined as follows. For each $t \in [0,1]$, write $\alpha : [0,1] \to M$ as the product
$$
\alpha = \alpha_t \alpha^t,
$$
where $\alpha_t$ is the restriction of $\alpha$ to $[0,t]$ and $\alpha^t$ is its restriction to $[t,1]$.\footnote{Strictly speaking, we should reparameterize each so that it is defined on $[0,1]$, but since the value of an iterated line integral on a path is independent of its parameterization, this is not necessary.} Then $\alphahat : [0,1] \to \Lambda M$ is the loop $\alphatilde(t) = \alpha^t \alpha_t$ in $\Lambda_{\alpha(t)} M$. The Chas--Sullivan map \cite{chas-sullivan}
$$
\cs : H_0(\Lambda M) \to H_1(M;\bL_M)
$$
takes $[\alpha]$ to $[\alphahat]$.

There is a similar construction for path spaces due to Kawazumi and Kuno. Each $\gamma \in P_{x_0,x_1} M$ gives rise to a section $\gammahat$ of $\gamma^\ast\bP_{M;x_0,x_1}$. It is induced by the lift $\gammatilde$
$$
\xymatrix{
& P_{x_0,x_1} M \ar[d]^p \cr
[0,1] \ar[ur]^{\gammatilde} \ar[r]_\gamma & M
}
$$
of $\gamma$ defined as follows. For each $t \in [0,1]$, write $\gamma : [0,1] \to M$ as the product $\gamma = \gamma_t' \gamma_t''$ where $\gamma_t'$ is the restriction of $\gamma$ to $[0,t]$ and $\gamma_t''$ is its restriction to $[t,1]$. Then $\gammatilde : [0,1] \to P_{x_0,x_1} M$ is the lift of $\gamma$ defined by $\gammatilde(t) = \gamma_t'\gamma_t''$. It is a cycle relative to $\{x_0,x_1\}$. The Kawazumi--Kuno map
$$
\csp : H_0(P_{x_0,x_1} M) \to H_1(M,\{x_0,x_1\};\bP_{M;x_0,x_1})
$$
takes $\gamma$ to $[\gammahat]$.

\begin{remark}
The first map appears in the paper \cite[\S 7]{chas-sullivan} of Chas and Sullivan, where it was used to give a homological formula for the Goldman bracket. Both maps appear in the paper \cite[\S 3]{kk:log} of Kawazumi and Kuno, who rediscovered the first map and defined the second. They rediscovered the Chas--Sullivan formula for the Goldman bracket and used the second to give an analogous homological description of the action.
\end{remark}

\section{The Goldman Bracket and the KK-Action}
\label{sec:goldman}

Now suppose that $X$ is a compact, connected, oriented surface with (possibly empty) boundary $\partial X$. Suppose that $x_0,x_1\in X$. These may or may not be boundary points. Set $X' = X - (\{x_0,x_1\}\cup \partial X)$. Let $j : X'\hookrightarrow X$ denote the inclusion. As in the previous section, we fix a commutative ring $\kk$. All homology groups will have coefficients in $\kk$, and tensor products will be over $\kk$ unless otherwise noted.

\subsection{Intersection pairings}

Suppose that $\bV$ is a local system of $\kk$-modules over $X'$ and that $\bW$ is a local system of $\kk$-modules over $X$. The intersection pairing
$$
H_1(X';\bV)\otimes H_1(X,\{x_0,x_1\}\cup \partial X;\bW)
\to H_0(X';\bV\otimes j^\ast\bW)
$$
induces a well-defined pairing
$$
\bil : H_1(X';\bV)\otimes H_1(X,\{x_0,x_1\};\bW) \to H_0(X';\bV\otimes j^\ast\bW).
$$
Similarly, for local systems $\bV$ and $\bW$ of $\kk$-modules over $X$, the intersection pairing
$$
H_1(X;\bV)\otimes H_1(X,\partial X;\bW) \to H_0(X;\bV\otimes\bW)
$$
restricts to an intersection pairing
$$
\bil : H_1(X;\bV)\otimes H_1(X;\bW) \to H_0(X;\bV\otimes\bW).
$$

\subsection{The Goldman bracket}

As is well-known (cf.\ \cite{chas,kk:groupoid}), the Goldman bracket can be expressed in terms of the intersection pairing and the Chas--Sullivan map. It is the composite
\begin{multline}
\label{eqn:bracket}
H_0(\Lambda(X))^{\otimes 2}
\overset{\cs^{\otimes 2}}{\To}
H_1(X;\bL_X)^{\otimes 2} \overset{\bil}{\To} H_0(X;\bL_X^{\otimes 2})
\cr
\overset{\mu_\ast}{\To} H_0(X;\bL_X) = H_0(\Lambda(X)),
\end{multline}
where the last map is induced by the multiplication pairing $\mu:\bL_X\otimes \bL_X \to \bL_X$.

\subsection{The Kawazumi--Kuno action}

Likewise, the Kawazumi--Kuno action can be expressed in terms of the intersection pairing. Let $j : X' \hookrightarrow X$ denote the inclusion.

\begin{proposition}
The KK-action
$$
\kappa : H_0(\Lambda X')\otimes H_0(P_{x_0,x_1}X) \to H_0(P_{x_0,x_1}X)
$$
is the composite
\begin{multline}
\label{eqn:action}
H_0(\Lambda(X')) \otimes H_0(P_{x_0,x_1}X) \overset{\cs\otimes\csp}{\To}
H_1(X';\bL_{X'})\otimes H_1(X,\{x_0,x_1\};\bP_{X;x_0,x_1})
\cr
\to H_0(X';\bL_{X'}\otimes j^\ast\bP_{X;x_0,x_1}) 
\overset{\mu_\ast}{\To} H_0(X';j^\ast\bP_{X;x_0,x_1})
\cong H_0(P_{x_0,x_1}X),
\end{multline}
where $\mu : \bL_{X'}\otimes j^\ast\bP_{X;x_0,x_1} \to j^\ast\bP_{X;x_0,x_1}$
denotes the natural pairing.
\end{proposition}

When $x_0,x_1 \in \partial X$, this action descends to the pairing
$$
H_0(\Lambda X)\otimes H_0(P_{x_0,x_1}X) \to H_0(P_{x_0,x_1}X).
$$
When $x_0=x_1$, it further descends to the Goldman bracket
$$
H_0(\Lambda X) \otimes H_0(\Lambda X) \to H_0(\Lambda X).
$$

\section{Unipotent Completion and Continuous Duals}
\label{sec:topology}

Suppose that $\cC$ is a $\kk$-linear abelian category with a faithful functor $\omega : \cC \to \Vec_\kk$ to the category of $\kk$-vector spaces. We also require that $\cC$ has a unit object $\unit$ with $\omega(\unit)$ a one dimensional vector space and that $\cC$ is closed under direct and inverse limits. Define the dimension of an object $V$ of $\cC$ to be $\dim_\kk \omega(V)$.

The two such categories relevant in this paper are:
\begin{enumerate}

\item The category of $\kk$-local systems over a path connected topological space $X$, where $\omega$ takes a local system to its fiber over a fixed base point $x\in X$. The unit object $\unit$ is the trivial local system of rank 1.

\item The category of representations over $\kk$ of a group $\Gamma$ and $\omega$ takes a representation to its underlying vector space. The unit object is the trivial 1-dimensional representation.

\end{enumerate}
These are related as the category of $\kk$-local systems over $X$ (path connected and locally contractible) is equivalent to the category of representations of $\pi_1(X,x)$ over $\kk$. The equivalence takes a local system to its fiber over the base point $x$.

An object $V$ of $\cC$ is said to be {\em trivial} if it is isomorphic to a finite direct sum of unit objects.

\begin{definition}
An object $V$ of $\cC$ is {\em unipotent} if it has a finite filtration
$$
0 = V_0 \subseteq V_1 \subseteq \cdots \subseteq V_{r-1} \subseteq V_r = V
$$
in $\cC$, where each $V_j/V_{j-1}$ is trivial.
\end{definition}

Each object $V$ of $\cC$ has a natural topology, where the neighbourhoods of $0$ are the kernels of homomorphisms $V \to U$, where $U$ is a unipotent object of $\cC$.

\begin{definition}
The {\em unipotent completion} $V^\un$ of an object $V$ of $\cC$ is its completion in this topology. In concrete terms
$$
V^\un = \varprojlim_{\substack{V/K \cr\text{unipotent}}} V/K
$$

\end{definition}

The {\em continuous dual} $\check{V}$ of a pro-object $V$ of $\cC$ is
$$
\check{V} = \Hom^\cts_\cC(V,\unit) :=
\varinjlim_{\substack{V/K \cr\text{unipotent}}}\Hom_\cC(V/K,\unit).
$$

If $\G$ is a discrete group and $H_1(\G;\kk)$ is finite dimensional, then the unipotent completion of the group algebra $\kk\G$ is its $I$-adic completion
$$
\kk\G^\wedge := \varprojlim_{n\ge 0} \kk\G/I^n,
$$
where $I$ is the augmentation ideal. This is a complete Hopf algebra. In this case, the {\em unipotent completion} (sometimes called the {\em Malcev completion}) $\G^\un_{/\kk}$ of $\G$ over $\kk$ is the group-like elements of $\kk\G^\wedge$ and its Lie algebra is the set of primitive elements of $\kk\G^\wedge$.\footnote{More precisely, the unipotent completion $\G^\un_{/\kk}$ over $\kk$ of an arbitrary discrete group $\G$ is the affine $\kk$-group whose coordinate ring is the Hopf algebra $\Hom^\cts(\kk\G,\kk)$. When $H_1(\G;\kk)$ is finite dimensional its group of $A$-rational points, where $A$ is a $\kk$-algebra, is the set of group-like elements of $A\G^\wedge$.} When $\kk = \Q$, we will write $\G^\un$ in place of $\G^\un_{/\Q}$. The {\em unipotent fundamental group} $\pi_1^\un(X,x)_{/\kk}$ of a pointed topological space is defined to be the unipotent completion of $\pi_1(X,x)$ over $\kk$. In this case the unipotent completion of a finitely generated $\G$-module $V$ (or one with $V/IV$ finite dimensional) is its $I$-adic completion:
$$
V^\un = V^\wedge := \varprojlim_n V/I^n V.
$$

The category $\Rep^\un_\kk(\G)$ of unipotent representations of $\G$ over $\kk$ is equivalent to the category of representations of $\G^\un$. Consequently, there are isomorphisms
$$
\Ext^\bdot_{\Rep^\un_\kk(\G)}(\kk,V)
\cong H^\bdot(\G^\un_{/\kk},V) \cong H^\bdot_\cts(\p,V),
$$
where $\p$ is the Lie algebra of $\G^\un_{/\kk}$ and $H^\bdot_\cts(\p,V)$ is its continuous cohomology.

\subsection{Rational $K(\pi,1)$ spaces}

The natural homomorphism $\G \to \G^\un(\Q)$ induces a homomorphism
$$
H^\bdot(\G^\un;\Q) \to H^\bdot(\G;\Q).
$$
When $\G = \pi_1(M,x)$, it can be composed with $H^\bdot(\pi_1(M,x)) \to H^\bdot(M)$, the canonical homomorphism, to get a natural homomorphism
\begin{equation}
\label{eqn:canon_homom}
H^\bdot(\pi_1^\un(M,x);\Q) \to H^\bdot(M;\Q)
\end{equation}

\begin{definition}
A connected topological space $M$ is a rational $K(\pi,1)$ if the homomorphism (\ref{eqn:canon_homom}) is an isomorphism.
\end{definition}

An equivalent version of the definition is that $M$ is a rational $K(\pi,1)$ if and only if its Sullivan minimal model is generated in degree 1. The equivalence of the definitions corresponds to the fact that that the 1-minimal model of $M$ is isomorphic to the DGA $\Hom^\cts(\Lambda^\bdot\p,\kk)$ of continuous Chevalley--Eilenberg cochains on $\p$.

\begin{theorem}
\label{thm:surfaces}
Every connected oriented surface with finite topology, except for the 2-sphere, is a rational $K(\pi,1)$. Equivalently, every oriented surface with non-positive Euler characteristic is a rational $K(\pi,1)$.
\end{theorem}

\begin{proof}
When $\pi$ is free, the result follows trivially from the fact that the cohomology of a free Lie algebra vanishes in degrees $>1$. When $\pi$ is the fundamental group of a genus $g$ surface, this seems to be well-known folklore.\footnote{See, for example, pages 290 and 316 of \cite{sullivan}.} However, I will include a very brief proof as I do not know any good references.

The starting point is that a compact Riemann surface $X$, being a compact K\"ahler manifold, is ``formal''.\footnote{Alternatively, this follows from the existence of a MHS on $\p$ and the purity of $H_1(X)$.} This implies that the Lie algebra $\p$ of the unipotent completion of $\pi = \pi_1(X,x)$ is isomorphic to the completion of its associated graded:
$$
\p \cong [\L(H)/(\theta)]^\wedge,
$$
where $H=H_1(X;\Q)$ and $\theta = \sum_{j=1}^g [a_j,b_j]$ with $a_1,\dots,b_g$ a symplectic basis of $H$. The continuous cohomology of $\p$ is isomorphic to the cohomology of the associated graded Lie algebra $\p_\bdot := \L(H)/(\theta)$. It is clear that $H_0(\p_\bdot) = \Q$, $H_1(\p_\bdot) \cong H$. To prove the result, it suffices to show that $H_j(\p_\bdot)$ is is 1-dimensional when $j=2$ and vanishes when $j>2$.

The universal enveloping algebra $U\p_\bdot$ is graded isomorphic to $A_\bdot := T(H)/(\theta)$, where $T(H)$ denotes the tensor algebra on $H$. The homological assertion is proved by observing that, for all $n\ge 0$,
$$
\xymatrix{
0 \ar[r] & A_n \ar[r]^(.38){\partial_2} & H\otimes A_{n+1} \ar[r]^(.56){\partial_1} & A_{n+2} \ar[r] &  0
}
$$
is exact, where $\partial_2(u) = \sum_{j=1}^g (a_j \otimes b_ju - b_j\otimes a_ju)$ and $\partial_1(x\otimes v) = xv$. It implies that
$$
\xymatrix{
0 \ar[r] & A_\bdot \ar[r]^(.4){\partial_2} & H\otimes A_\bdot \ar[r]^(.58){\partial_1} & A_\bdot \ar[r] &  \Q \ar[r] & 0
}
$$
is a left $A_\bdot$ resolution of the trivial $A_\bdot$-module $\Q$ from which the homological computation follows.
\end{proof}

\section{Iterated Integrals and Path Space De~Rham Theorems}
\label{sec:it_ints}

We will assume that the reader has a basic familiarity with Chen's iterated integrals. Three relevant references are Chen's comprehensive survey article \cite{chen:bams}, the introduction \cite{hain:bowdoin} to iterated integrals and Hodge theory, and \cite{hain:dht}, in which the mixed Hodge structures on the cohomology of path spaces used in this paper are constructed. Throughout this section, $\kk$ will be either $\R$ or $\C$.\footnote{It is possible to take $\kk = \Q$ if ones take $E^\bdot(X)$ to be the DGA of Sullivan's rational PL forms on the simplicial set of singular simplices on a topological space $X$.} The complex of smooth $\kk$-valued forms on a manifold $M$ (or differentiable space) will be denoted by $E^\bdot_\kk(M)$.

We begin with a very quick overview of Chen's theory. Suppose that $M$ is a smooth manifold. One needs a notion of differential forms on $PM$. Chen uses his useful and elementary notion of differential forms on a ``differentiable space''. But iterated integrals can be defined for any reasonable notion of differential forms on $PM$.

The {\em time ordered $r$-simplex} is
$$
\Delta^r :=
\{(t_1,\dots,t_r) \in \R^r : 0 \le t_1 \le \cdots \le t_r \le 1\}.
$$
For each $r\ge 0$, there is a ``sampling map''
$\phi_r : \Delta^r \times PM \to M \times M^r \times M$. It is defined by
$$
\phi_r(t_1,\dots,t_r,\gamma) =
\big(\gamma(0),\gamma(t_1),\dots,\gamma(t_r),\gamma(1)\big).
$$
Now suppose that $w_1,\dots,w_r \in E^\bdot(M)$. Then (up to a sign, which depends on conventions)
$$
\int(w_1|w_2|\dots|w_r)
$$
is the differential form $\pi_\ast \phi_r^\ast(1\times w_1 \times \dots \times w_r \times 1)$ on $PM$, where $\pi : \Delta^r \times PM \to PM$ is the projection. It has degree $\sum_j(-1+\deg w_j)$ and vanishes when any of the $w_j$ is of degree 0. In particular, when each $w_j$ is a 1-form, then $\int(w_1|\dots|w_r)$ is a smooth function on $PM$. Its value on $\gamma \in PM$ is the time ordered integral
$$
\int_\gamma (w_1|\dots| w_r) := \int_{\Delta^r} w_1 \times \dots \times w_r.
$$
Its value on a path does not depend on the parameterization of $\gamma$. We will use Chen's sign conventions (used in both \cite{chen:bams} and \cite{hain:dht}).

The space of iterated integrals is the subspace of $E^\bdot(PM)$ spanned by the elements
$$
p_0^\ast w' \wedge \int(w_1|\dots|w_r) \wedge p_1^\ast w''
$$
where $w',w'' \in E^\bdot(M)$. The space of iterated integrals is closed under exterior derivative and wedge product and is therefore a DGA.

\subsection{Reduced bar construction}

Iterated integrals on $PM$ and on ``pullback path fibrations'' are described algebraically by Chen's {\em reduced bar construction}. The following discussion summarizes results in Chen \cite[Chapt.~IV]{chen:bams}, especially in Section~4.2. There are also relevant discussions in Sections 1, 2 and 3.4 of \cite{hain:dht}.

Suppose that $f : N \to M\times M$ is a smooth map of smooth manifolds. One can pullback the free path fibration $p_0\times p_1: PM \to M\times M$ to obtain a fibration $p : P_f M \to N$:
\begin{equation}
\label{eqn:pullback}
\xymatrix{
P_f M \ar[r]^{Pf} \ar[d]_p \ar[r] & PM \ar[d]^{p_0\times p_1} \cr
N \ar[r]^f & M\times M
}
\end{equation}
The total space $P_fM$ is naturally a differentiable space. Two relevant cases are:
\begin{enumerate}

\item $N = \{x_0,x_1\}$ and $f$ is the inclusion, in which case $P_f M$ is $P_{x_0,x_1} M$;

\item $N=M$ and $f$ is the diagonal, in which case $P_f M$ is $\Lambda M$.

\end{enumerate}

Pullbacks of iterated integrals on $PM$ along $Pf$ and pullbacks of forms on $N$ along $p$ generate a sub-DGA of $E^\bdot(P_f M)$, which we denote by $\Ch^\bdot(P_f M)$. It admits an algebraic description in terms of the reduced bar construction. We review a special case.

Suppose that $A^\bdot$, $A_0^\bdot$ and $A_1^\bdot$ are non-negatively graded DGAs over $\kk$. Suppose that 
$$
\epsilon_j : A^\bdot \to A^\bdot_j,\qquad j=0,1
$$
are DGA homomorphisms. Then one can define (\cite[\S~4.1]{chen:bams}, \cite[\S~1.2]{hain:dht}) the reduced bar construction
$$
B(A_0^\bdot,A^\bdot,A_1^\bdot).
$$
With these assumptions, it is a non-negatively graded commutative DGA.

It is the $\kk$-span of elements $a'[a_1|\dots|a_r]a''$ which are multilinear in the $a$'s, where $a'\in A^\bdot_0$, $a''\in A^\bdot_1$ and each $a_j \in A^+$.\footnote{The empty symbol $[\blank]$, which corresponds to $r=0$, is interpreted as $1$.} It is an $A_0^\bdot \otimes A_1^\bdot$-module.

For future reference, we note that if $a_1,\dots a_r \in A^1$, then
\begin{multline}
\label{eqn:diff}
d[a_1|\dots|a_r] = 
-\sum_{j=1}^r [a_1|\dots|d a_j| \dots |a_r]
-\sum_{j=1}^{r-1} [a_1|\dots|a_j\wedge a_{j+1}|\dots|a_r]
\cr
+ [a_1|\dots|a_{r-1}]a_r - a_1[a_2|\dots|a_r].
\end{multline}
In addition, for each $f\in A^0$, there are also the relations
\begin{equation}
\label{eqn:relns}
\begin{aligned}[b]
[a_1|\dots|a_{j-1}|df|a_j|\dots|a_r]
&= [a_1|\dots |a_{j-1}|fa_j|\dots|a_r] - [a_1|\dots |fa_{j-1}|a_j|\dots|a_r]
\cr
[df|a_1\dots|a_r] &= [fa_1|\dots|a_r] - f[a_1|\dots|a_r]
\cr
[a_1|\dots|a_r|df] &= [a_1|\dots|fa_r] - [a_1|\dots|a_r]f
\cr 
[df] &= [\blank]f - f[\blank]
\end{aligned}
\end{equation}
where $r>0$ in the first three relations and $j>1$ in the first relation.

The map
$$
B\big(E^\bdot(M),E^\bdot(M),E^\bdot(M)\big)
\overset{\simeq}{\To} \Ch^\bdot(PM),
$$
defined by
$$
w'[w_1|\dots|w_r]w'' \mapsto p_0^\ast w' \wedge \int(w_1|\dots|w_r) \wedge p_1^\ast w''
$$
where $\e_0$ and $\e_1$ are both the identity, is a well-defined DGA isomorphism.

If $C^\bdot$ is a commutative DGA and $\e : A^\bdot \otimes A^\bdot \to C^\bdot$ is a DGA homomorphism, one defines the (reduced) circular bar construction by
$$
B(A^\bdot;C^\bdot) := B(A^\bdot,A^\bdot,A^\bdot)\otimes_{(A^\bdot\otimes A^\bdot)} C^\bdot.
$$
It is spanned by elements of the form $[a_1|\dots|a_r]c$ where each $a_j \in A^+$ and $c\in C^\bdot$.

A smooth map $f : N \to M\times M$ induces a homomorphism
$$
f^\ast : E^\bdot(M)\otimes E^\bdot(M) \to E^\bdot(N).
$$
We can thus form the circular bar construction $B(E^\bdot(M);E^\bdot(N))$. In terms of the notation above, the map
$$
B(E^\bdot(M);E^\bdot(N)) \to \Ch^\bdot(P_f M)
$$
that takes
$$
[w_1|\dots|w_r]w \mapsto (Pf)^\ast \Big(\int(w_1|\dots | w_r)\Big) \wedge p^\ast w
$$
is a well-define DGA isomorphism. See \cite[Thm.~4.2.1]{chen:bams}.

It is useful to know that complexes of iterated integrals constructed from quasi-isomorphic sub-DGAs of $E^\bdot(M)$ and $E^\bdot(M)$ are quasi-isomorphic.

\begin{proposition}[{Chen, cf.\ \cite[Cor.~1.2.3]{hain:dht}}]
\label{prop:qism}
Suppose that
$$
\xymatrix{
A_1^\bdot \otimes A_1^\bdot \ar[r]^{e_1}\ar[d]_{\phi_A} & C_1^\bdot \ar[d]^{\phi_C} \cr
A_2^\bdot \otimes A_2^\bdot \ar[r]^{\e_2} & C^\bdot_2
}
$$
is a diagram of DGA homomorphisms in which $A^\bdot_j$ and $C^\bdot_j$ are commutative and non-negatively graded ($j=1,2$). If $\phi_A$ and $\phi_C$ are quasi-isomorphisms and $A^\bdot_j$ are homologically connected, then the induced DGA homomorphism
$$
\phi : B(A^\bdot_1;C_1^\bdot) \To B(A^\bdot_2;C_2^\bdot)
$$
is a quasi-isomorphism.
\end{proposition}

\begin{corollary}
\label{cor:stalk}
If $g : Y \to N$ is a homotopy equivalence, then
$$
(Ph)^\ast : \Ch^\bdot(P_f M) \to \Ch^\bdot(P_{f\circ g}M)
$$
is a quasi-isomorphism.
\end{corollary}

\begin{remark}
\label{rem:qism}
If $A^\bdot$ is a non-negatively graded DGA with $A^0 = \kk$, then the associated reduced bar and circular bar constructions are canonically graded. Specifically, there are canonical isomorphisms
$$
B(A_1^\bdot;A^\bdot,A_2^\bdot) \cong \bigoplus_{n\ge 0} A_1^\bdot \otimes (A^+)^{\otimes n} \otimes A_2^\bdot
$$
and
$$
B(A^\bdot;C^\bdot) \cong \bigoplus_{n\ge 0} (A^+)^{\otimes n} \otimes C^\bdot
$$
where $A^+ = \oplus_{n>0} A^n$. This fact, in conjunction with Proposition~\ref{prop:qism}, is a useful technical tool. One can replace $E^\bdot(M)$ by quasi-isomorphic sub-DGA $A^\bdot$ with $A^\bdot = \kk$. This is particularly useful when doing relative de~Rham theory, as in Sections~\ref{sec:pullback_fibrations} and \ref{sec:rel-DR-PM}.
\end{remark}

\subsection{De~Rham theorems}

Denote the cohomology of the complex of $\kk$-valued iterated integrals on the pullback path fibration $P_f M$ by
$$
H^\bdot_\dR(P_f M;\kk).
$$

\begin{theorem}[Chen]
\label{thm:chen_dr}
If $M$ is connected, then integration induces an isomorphism
$$
H^0_\dR(\Lambda_x M;\kk) \overset{\simeq}{\To}
\varinjlim_n \Hom_\kk(\kk\pi_1(M,x)/I^n,\kk).
$$
Consequently, if $H^1(M;\kk)$ is finite dimensional, then for each $x\in M$ integration induces a Hopf algebra isomorphism
$$
\cO(\pi_1^\un(M,x)_{/\kk}) \cong H^0_\dR(\Lambda_x M;\kk).
$$
\end{theorem}

\begin{corollary}
\label{cor:path-dr}
If $H_1(M;\kk)$ is finite dimensional, then for all $x_0,x_1\in M$, integration induces an isomorphism
$$
H^0_\dR(P_{x_0,x_1} M;\kk) \cong \Hom^\cts(H_0(P_{x_0,x_1}M),\kk).
$$
\end{corollary}

\subsection{Iterated integrals and rational $K(\pi,1)$ spaces}

Since the loop space $\Lambda_x X$ of a path connected topological space is a disjoint union of contractible spaces, a space $X$ is a $K(\pi,1)$ if and only if $H^j(\Lambda_x X)$ vanishes for all $j>0$. There is an analogous statement for rational $K(\pi,1)$. It is an immediate consequence of \cite[Thm.~2.6.2]{hain:dht} and the definition of a rational $K(\pi,1)$.

\begin{proposition}
\label{prop:rat_Kpi1}
A connected manifold $M$ is a rational $K(\pi,1)$ if and only if $H^j_\dR(\Lambda_x M)$ vanishes when $j>0$ for one (and hence all) $x\in M$. Equivalently, it is a rational $K(\pi,1)$ if and only if $H^j_\dR(P_{x_0,x_1}M) = 0$ when $j>0$ for all $x_0,x_1\in M$.
\end{proposition}

\subsection{Relative de~Rham Theory for pullback path fibrations}
\label{sec:pullback_fibrations}

Suppose that $M$ is a connected manifold. Fix a smooth map $f : N \to M\times M$. Then one has the pullback path fibration (\ref{eqn:pullback}). Chen \cite{chen:bams} proved that when $M$ is simply connected, integration induces a graded algebra isomorphism
$$
H^\bdot_\dR(P_f M;\R) \to H^\bdot(P_f M;\R).
$$
Here we investigate what happens when $M$ is connected, but not simply connected.\footnote{These results do not seem to have previously appeared in the literature.} We are particularly interested in the case where $P_f M$ is the free loop space $\Lambda M$. In the current context, it is best to do this using sheaf theory. For simplicity, we assume that $M$ and $N$ are finite complexes.

Denote the restriction of the fibration $P_f M \to M$ to an open subset $U$ of $N$ by $P_f M|_U$. Let $\Pdual_f^j$ be the sheaf over $N$ associated to the presheaf
$$
U \mapsto H^j_\dR(P_f M|_U;\kk).
$$
Corollary~\ref{cor:stalk}, applied to the inclusion $\{x\}\to U$, implies that it is locally constant. The finiteness assumptions imply that it is a direct limit of unipotent local systems.

When $M$ is a rational $K(\pi,1)$, $\Pdual_f^j$ will vanish when $j>0$. In this case, we will sometimes write $\Pdual_f$ instead of $\Pdual_f^0$.

Denote the sheaf of smooth $\kk$-valued $k$-forms on $N$ (a manifold) by $\sE_N^k$. Define
$$
\sP_f^k := \sE_N^0 \otimes_\kk \Pdual_f^k.
$$
It is a direct limit of flat vector bundles.

\begin{proposition}
\label{prop:ss}
There is a spectral sequence converging to $H^{j+k}_\dR(P_f M;\kk)$ with $E_1^{j,k} = E^j(N;\sP_f^k)$ and $d_1 = \nabla$, so that $E_2^{j,k} = H^j(N;\sP^k_f)$.
\end{proposition}

\begin{proof}
To simplify the proof, we choose a sub-DGA $A^\bdot$ of $E^\bdot(M)$ with $A^0 = \kk$ and such that the inclusion $A^\bdot \hookrightarrow E^\bdot(M)$ is a quasi-isomorphism. Any will do. Proposition~\ref{prop:qism} implies that the inclusion
$$
B(A^\bdot;E^\bdot N) \hookrightarrow \Ch^\bdot(P_f M)
$$
induces an isomorphism
$$
H^\bdot(B(A^\bdot;E^\bdot N)) \cong H^\bdot_\dR(P_f M;\kk).
$$
The first benefit of replacing $E^\bdot(M)$ by $A^\bdot$ is that for all submanifolds $U$ of $N$ (such as $U$ open or a point), there is a natural isomorphism
\begin{equation}
\label{eqn:gr_isom}
B(A^\bdot;E^\bdot(U)) \cong B(\kk,A^\bdot,\kk) \otimes E^\bdot(U)
\end{equation}
of graded vector spaces. So, unlike $\Ch^\bdot(P_f M|_U)$ which is only a filtered complex, $B(A^\bdot;E^\bdot(U))$ is a double complex. This will simplify the homological algebra. The second benefit is that, since the restriction of the augmentation $\e_x$ to $A^\bdot$ is independent of $x\in M$, the inclusion
$$
B(A^\bdot;\kk) \cong B(\kk,A^\bdot,\kk) \hookrightarrow \Ch^\bdot(P_{x,y} M;\kk)
$$
is a quasi-isomorphism for all $x,y\in M$. This gives a trivialization of the vector bundles $\sP_f^k$ for all $k\ge 0$; the trivializing sections are those in $H^k(B(A^\bdot;\kk))\otimes \kk$.

Consider the sheaf $\sB_f^\bdot$ of complexes on $N$ associated to the presheaf
$$
U \mapsto B(A^\bdot;E^\bdot(U)).
$$
Corollary~\ref{cor:stalk} and Corollary~\ref{cor:path-dr} imply that the associated homology sheaf is the graded sheaf $\Pdual_f^\bdot\otimes_\kk \sE^\bdot_N$. It has a natural connection, which we denote by $\nabla$.

The complex $B(A^\bdot;E^\bdot(N))$ can be filtered by degree in $E^\bdot(N)$. The corresponding spectral sequence converges to
$$
H^\bdot(B(A^\bdot;E^\bdot(N))) = H^\bdot_\dR(P_f M;\kk).
$$
It satisfies
$$
E_1^{j,k} = E^j(N;\sP_f^k) = H^k(B(A^\bdot;\kk))\otimes E^j(N)
$$
and has $E_2$ term $E_2^{j,k} = H^j(N;\sP_f^k)$.

To complete the proof, we need to show that the differential $d_1$ coming from the double complex $B(A^\bdot;E^\bdot(N))$ equals the differential induced by the flat connection $\nabla$ on $\sP_f^\bdot$. We do this when $k=0$, which is all we shall need, as surfaces with non-positive Euler characteristic are rational $K(\pi,1)$ spaces. The proof in the case $k>0$ is similar, but more technical as it uses the definition of iterated integrals of higher degree.

In the case $k=0$, the equality of $d_1$ and $\nabla$ reduces to the Fundamental Theorem of Calculus as follows: Suppose that if $w_1,\dots,w_r \in A^1$ and that $\gamma \in PM$. For $0\le s < t \le 1$, denote the restriction of $\gamma$ to $[s,t]$ by $\gamma_s^t$. By the Fundamental Theorem of Calculus, the the value of the exterior derivative of the function
$$
(s,t) \mapsto \int_{\gamma_s^t} (w_1|\dots|w_r)
$$
at $(0,1)$ is
$$
\langle w_r,\dot{\gamma}(1)\rangle\int_\gamma (w_1|\dots|w_{r-1}) 
- \langle w_1,\dot{\gamma}(0)\rangle \int_\gamma (w_2|\dots|w_r).
$$
That $d_1=\nabla$ follows as every section of $\sP_f$ is an $\sE^0_M$-linear combination of elements of $H^0(B(A^\bdot;\kk))$ and these, in turn, are $\kk$-linear combination of the sections of $\sB^0$ considered above.
\end{proof}

Proposition~\ref{prop:rat_Kpi1} implies that $\sP^j_f$ vanishes for all $j>0$ when $M$ is a rational $K(\pi,1)$.

\begin{corollary}
\label{cor:vanishing}
For all manifolds $M$ with finite dimensional $H^1(M;\kk)$, we have $H^0_\dR(P_f M;\kk) \cong H^0(M;\sP_f^0)$. If $M$ is a rational $K(\pi,1)$, then for all $j\ge 0$
$$
H^j_\dR(P_f M;\kk) \cong H^j(M;\sP_f^0).
$$
\end{corollary}

The second statement is the analogue of Proposition~\ref{prop:k_pi_1} for rational $K(\pi,1)$ spaces.

These results can be assembled into a de~Rham theorem for the degree 0 cohomology of pullback path spaces, such as the free loop space of $M$. If $x\in N$, then $\pi_1(N,x)$ acts naturally on $\pi(M;y_0,y_1)$, where $f(x) = (y_0,y_1)$; it acts by left multiplication via the homomorphism $\pi_1(N,x) \to \pi_1(M,y_0)$ and right by multiplication by the inverse of $\pi_1(N,x) \to \pi_1(M,y_1)$.

\begin{theorem}
\label{thm:pullback_DR}
If $N$ is connected and $M$ has finite dimensional $H^1(M;\kk)$, then
$$
H^0_\dR(P_f M;\kk) \cong
\Hom^\cts_\kk(\kk\pi(M;y_0,y_1),\kk)^{\pi_1(N,x)}.
$$
In particular, when $f$ is the diagonal map $\Delta : M \to M\times M$, we have
$$
H^0_\dR(\Lambda M;\kk) \cong H^0(M;\sL_M) \cong \Hom^\cts_\kk(\kk\pi_1(M,x),\kk)^{\pi_1(M,x)}
$$
where $\pi_1(M,x)$ acts on itself by conjugation and where $\sL_M := \sP_\Delta^0$, which has fiber $H^0_\dR(\Lambda_x M;\kk)$ over $x\in M$.
\end{theorem}

\begin{corollary}
If $M$ is a rational $K(\pi,1)$ with $\dim H^1(M;\kk) < \infty$, then
$$
H^j_\dR(\Lambda M;\kk) \cong H^j(M;\sL_M)
$$
for all $j\ge 0$.
\end{corollary}

\section{Notational Interlude}

Suppose that $M$ is a smooth manifold. Give $H_0(P_{x_0,x_1}M;\kk)$ the unipotent topology and $H_0(\Lambda M;\kk)$ the quotient topology induced by the canonical surjection $H_0(P_{x,x} M;\kk) \to H_0(\Lambda M;\kk)$. Set
$$
\Hdual^0(P_{x_0,x_1}M;\kk) := \Hom^\cts_\kk(H_0(P_{x_0,x_1}M),\kk)
$$
and
$$
\Hdual^0(\Lambda M;\kk) := \Hom^\cts_\kk(H_0(\Lambda M),\kk).
$$
If $H_1(M;\kk)$ is finite dimensional, then the $I$-adic and unipotent topologies on $\kk\pi_1(M,x)$ coincide. This implies that the $I$-adic and unipotent completions of $H_0(P_{x_0,x_1}M;\kk)$ and $H_0(\Lambda M;\kk)$ also coincide. In this case (and with $\kk = \R$ or $\C$), the path space de~Rham theorems above assert that integration induces isomorphisms
$$
H_\dR^0(P_{x_0,x_1} M;\kk) \overset{\simeq}{\To} \Hdual^0(P_{x_0,x_1} M;\kk)
\text{ and }
H_\dR^0(\Lambda M;\kk) \overset{\simeq}{\To} \Hdual^0(\Lambda M;\kk).
$$
In subsequent sections, the manifolds $M$ considered have finitely generated fundamental groups. In particular, they satisfy the condition that $H_1(M;\kk)$ is finite dimensional.

\section{Iterated Integrals and the Goldman Bracket}
\label{sec:goldman_DR}

In this section we show that the continuous dual of the Goldman bracket is the composition of maps, each of which can be expressed in terms of Poincar\'e duality or iterated integrals. This is essentially dual to to the factorization (\ref{eqn:bracket}) of the Goldman bracket. In this section, we fix $\kk$ to be $\R$ or $\C$.

\subsection{A formula for $\cs$}

In this section we give a formula for the continuous dual of $\cs$, which holds for all smooth manifolds $M$. For simplicity, we suppose that $M$ has finite Betti numbers. Recall the definition of $\cs$ from Section~\ref{sec:CS}.

Denote the continuous dual of $\bL_M$ by $\Ldual_M$. Chen's de~Rham Theorem (Thm.~\ref{thm:chen_dr}) implies that it is the locally constant sheaf associated to the flat bundle $\sL_M$.

Set $\bL_M^\ast = \Hom_\Z(\bL_M,\kk)$. Then $\cs$ induces a map
$$
\cs^\ast : H^1(M;\bL_M^\ast) \cong \Hom(H_1(M;\bL_M),\kk) \to H^0(\Lambda M;\bL_M^\ast).
$$
The natural inclusion $\Ldual_M \hookrightarrow \bL_M^\ast$ induces a map
$$
\theta : H^1(M;\sL_M) \cong H^1(M;\Ldual_M) \to H^1(M;\bL_M^\ast).
$$

\begin{proposition}
\label{prop:formula_bcs}
There is a map
$$
\csdual : H^1(M;\sL_M) \to H^0_\dR(\Lambda M)
$$
which is dual to $\cs$. That is, the diagram
$$
\xymatrix{
H^1(M;\sL_M) \ar[r]^\csdual\ar[d]_\theta & H^0_\dR(\Lambda M;\kk) \ar[d]^{\text{integration}} \cr
H^1(M;\bL_M^\ast) \ar[r]^{\cs^\ast} & H^0(\Lambda M;\kk)
}
$$
commutes. In terms of the complex $E^\bdot(M;\sL_M)$, the map $\csdual$ is induced by
$$
\csdual : [w_1|\dots|w_r]\otimes w \mapsto \sum_{j=0}^r [w_{j+1}|\dots|w_r|w|w_1|\dots|w_j],
$$
where $w,w_1,\dots,w_r \in E^1(M)$.
\end{proposition}

\begin{proof}
Proposition~\ref{prop:ss} implies that elements of $E^1(M;\sL_M)$ are linear combinations $\sum_k I_k \otimes w_k$ where $w_k \in E^1(M)$ and
$$
I_k \in \Gr H^0\big( B(E^\bdot(M);E^\bdot(M))/([\blank]\otimes E^\bdot(M))\big).
$$
These, in turn, are linear combinations of expressions
$$
\varphi = [w_1|\dots|w_r]\otimes w
$$
where $w$ and each $w_j$ is in $E^1(M)$.

If $\alpha \in \Lambda M$, then the value of $\varphi$ on $\alphahat$ is
$$
\langle \varphi,\alphahat \rangle = \int_\alpha \langle \int(w_1|\dots|w_r), \alpha^t\alpha_t \rangle w.
$$
Let $f$ and $f_j$ be the smooth functions on the unit interval defined by
$$
\alpha^\ast w = f(t) dt \text{ and } \alpha^\ast w_j = f_j(t)dt,\qquad j=1,\dots, r.
$$
Since
\begin{align*}
F(t) : &= \langle \int(w_1|\dots|w_r), \alpha^t\alpha_t \rangle
\cr
&= \sum_{j=0}^r \int_{\alpha^t} (w_1|\dots|w_j) \int_{\alpha_t} (w_{j+1}|\dots|w_r)
\cr
&= \sum_{j=0}^r \ \ \idotsint\limits_{t\le t_1 \le \dots t_j \le 1} f_1(t_1) \dots f_j(t_j) \idotsint\limits_{0\le t_{j+1} \le \dots \le t_r \le t} f_{j+1}(t_{j+1}) \dots f_r(t_r)
\end{align*}
where all integrals are taken with respect to Lebesgue measure on each time ordered simplex. So

\begin{align*}
\langle \varphi,\alphahat \rangle
&=\int_0^1 f(t) F(t) dt
\cr
&= \sum_{j=0}^r\ \ \idotsint\limits_{0\le t_{j+1} \le \dots \le t_r \le t \le t_1 \le \dots t_j \le 1} f_{j+1}(t_{j+1}) \dots f_r(t_r) f(t) f_1(t_1) \dots f_j(t_j)
\cr
&= \sum_{j=0}^r \int_\alpha (w_{j+1}|\dots|w_r|w|w_1|\dots|w_j).
\end{align*}
The result follows as $\cs(\alpha) = [\alphahat]$.
\end{proof}

\begin{corollary}
The map $\cs : H_0(\Lambda M;\kk) \to H_1(M;\bL_M)$ is continuous in the $I$-adic topology and thus induces a map $\csdual : H^1(M;\Ldual_M) \to \Hdual^0(\Lambda M;\kk)$.
\end{corollary}

\subsection{The continuous dual of the Goldman bracket}

In this section, $X = \Xbar - Y$, where $\Xbar$ is a compact oriented surface (possibly with boundary) and $Y$ is the union of $\partial X$ and a finite set.\footnote{In a subsequent section, when we are applying Hodge theory, $\Xbar$ will be a compact Riemann surface and $Y$ will be finite.}

\begin{lemma}
\label{lem:cup_dual}
There is a map
\begin{equation}
\label{eqn:cup_dual}
\smile^\ast\ : H_2(\Xbar,Y;\Ldual_X) \to H_1(\Xbar,Y;\Ldual_X)^{\otimes 2}
\end{equation}
that is dual to the map
$$
\smile\ : H^1(\Xbar,Y;\bL_X)^{\otimes 2} \to H^2(\Xbar,Y;\bL_X)
$$
induced by the cup product and the fiberwise multiplication map $\bL_X^{\otimes 2} \to \bL_X$ via the natural pairings
$$
\bL_X \otimes \Ldual_X \to \kk_X
\text{ and }
H^\bdot(\Xbar,Y;\bL_X)\otimes H_\bdot(\Xbar,Y;\Ldual_X) \to \kk.
$$
\end{lemma}

\begin{proof}
The fiber of $\bL_X$ over $x\in X$ is $\kk \pi_1(X,x)$. Denote the local system corresponding to the $n$th power of the the augmentation ideals by $\bI^n$. Multiplication induces a map $\bL_X/\bI^n \otimes \bL_X/\bI^n \to \bL_X/\bI^n$.
Since $\dim H_1(\Xbar;\kk)< \infty$, each $\bL_X/\bI^n$ is finite dimensional and
$$
\Ldual_X = \varinjlim_n \Hom_\kk(\bL_X/\bI^n,\kk).
$$
Applying $\Hom_\kk(\blank,\kk)$ to the cup product
$$
H^1(\Xbar,Y;\bL_X/\bI^n)\otimes H^1(\Xbar,Y;\bL_X/\bI^n) \to H^2(\Xbar,Y;\bL_X/\bI^n)
$$
and then taking direct limits gives the map (\ref{eqn:cup_dual}).
\end{proof}

Poincar\'e duality is an isomorphism $H_j(\Xbar,Y;\bV) \to H^{2-j}(X;\bV)$ for all local systems $\bV$ over $X$. Note that, since $\sL_X$ is a direct limit of finite dimensional flat vector bundles, de~Rham's theorem induces an isomorphism
$$
H^\bdot(X;\sL_X) \cong H^\bdot(X;\Ldual_X).
$$
Consequently, there are Poincar\'e duality isomorphisms
$$
\PD : H_j(\Xbar,Y;\Ldual_X) \overset{\simeq}{\To} H^{2-j}(X;\sL_X).
$$

\begin{proposition}
\label{prop:dual_goldman}
The pairing $\Ldual_X \otimes \bL_X \to \kk_X$ induces a pairing of the composition of the top row, the right-hand vertical map and the second row of
$$
\xymatrix{
\Hdual^0(\Lambda X) \ar[r]^\simeq \ar[d]_{\gold^\vee} & H^0(X;\Ldual_X) \ar[r]^\simeq &
H_2(\Xbar,Y;\Ldual_X) \ar[d]^{\smile^\ast} &
\cr
\Hdual^0(\Lambda X)^{\otimes 2}& H^1(X;\Ldual_X)^{\otimes 2} 
\ar[l]_{\csdual^{\otimes 2}}
& H_1(\Xbar,Y;\Ldual_X)^{\otimes 2} \ar[l]_\simeq
}
$$
pairs with the factorization (\ref{eqn:bracket}) of the Goldman bracket. Consequently, the dual of the Goldman bracket is continuous.
\end{proposition}

The continuity of the Goldman bracket was proved directly by Kawazumi and Kuno in \cite[\S 4.1]{kk:groupoid}.

\subsection{Adams operations}
\label{sec:adams}

Even though $H_0(\Lambda M)$ is not a ring, one can take the power of any loop. We therefore have, for each $n\ge 0$, endomorphisms
$$
\psi_n : H_0(\Lambda M) \to H_0(\Lambda M)
$$
defined by taking the class $[\alpha]$ of a loop to the class $[\alpha^n]$ of its $n$th power. Denote the filtration on $H_0(\Lambda M)$ induced by the filtration of the group algebra by powers of its augmentation ideal by $I^\bdot$. The Adams operator $\psi_n$ has the property that $\psi_n(I^m)\subseteq I^{mn}$ and is therefore continuous in the $I$-adic topology. It therefore induces a map
$$
\psi_n : H_0(\Lambda M)^\wedge \to H_0(\Lambda M)^\wedge.
$$
In Section~\ref{sec:mhs_paths} we show that when $M$ is a smooth complex algebraic variety, each $\psi_n$ is a morphism of MHS.

\begin{remark}
Choose a base point $x\in M$. The set of primitive elements of $\Q\pi_1(M,x)^\wedge$ is the Lie algebra $\p(M,x)$ of the unipotent completion of $\pi_1(M,x)$ and $\Q\pi_1(M,x)^\wedge$ is the completed enveloping algebra of $\p(M,x)$. By the PBW Theorem, there is a canonical complete coalgebra isomorphism 
$$
\Q\pi_1(M,x)^\wedge \cong \prod_{k\ge 0} \Sym^k \p.
$$
The lift of $\psi_n$ to $\Q\pi_1(M,x)^\wedge$ takes $\gamma \to \gamma^n$ and is therefore multiplication by $n$ on $\p$.

But since the $n$th power map that takes $\gamma \in \pi_1(M,x)$ to $\gamma^n$ commutes with the coproduct $\Delta : \Q\pi_1(M,x)^\wedge \to [\Q\pi_1(M,x)^\wedge]^{\otimes 2}$, $\psi_n$ also commutes with $\Delta$. It follows that the restriction of $\psi_n$ to $\Sym^k \p(M,x)$ is multiplication by $n^k$. Since the surjection $\Q\pi_1(M,x)^\wedge \to H_0(\Lambda M;\Q)^\wedge$ commutes with the $\psi_n$, the PBW decomposition descends to a decomposition
$$
H_0(\Lambda M;\Q)^\wedge \cong \prod_{k\ge 0} \overline{S}^k(\Lambda M)
$$
where $\psi_n$ acts on $\overline{S}^k$ by multiplication by $n^k$.
\end{remark}

\section{De~Rham Theory and the Dual of the KK-Action}
\label{sec:kk_DR}

In this section we factor the continuous dual of of the KK-action as a composition of maps, each of which can be expressed in terms of Poincar\'e duality or iterated integrals.

\subsection{Relative de~Rham theory for $P_{x_0,x_1}M$}
\label{sec:rel-DR-PM}

Here we sketch the relative de~Rham theory for the fibration $p:P_{x_0,x_1}M \to M$ given by evaluation at $t=1/2$. It is similar to the relative de~Rham theory for pullback path fibrations discussed in Section~\ref{sec:pullback_fibrations}. For this reason, we will be brief. For simplicity, we suppose that $M$ is a manifold which is homotopy equivalent to a finite complex. Throughout, $\kk$ will be either $\R$ or $\C$. As in Section~\ref{sec:pullback_fibrations}, $\sE_M^j$ denotes the sheaf of smooth $\kk$-valued $j$-forms on $M$.

Denote the restriction of $p:P_{x_0,x_1}M\to M$ to the open subset $U$ of $M$ by $P_{x_0,x_1}M|_U$. Let $\Pdual_{x_0,x_1}^k$ be the locally constant sheaf associated to the presheaf
$$
U \mapsto H^k_\dR(P_{x_0,x_1}M|_U;\kk).
$$
It is a direct limit of unipotent local systems over $M$ and has fiber isomorphic to
$$
\Hom^\cts_\kk(\bP_{M;x_0,x_1}^k,\kk) \cong \bigoplus_{i+j=k}\Hom^\cts_\kk(H_i(P_{x_0,x}M)\otimes H_j(P_{x,x_1}M),\kk)
$$
over $x\in M$. When $M$ is a rational $K(\pi,1)$, $\Pdual^k_{x_0,x_1}$ vanishes when $k>0$. In this case, we will sometimes denote $\Pdual^0_{x_0,x_1}$ by $\Pdual_{x_0,x_1}$.

Define $\sP_{x_0,x_1}^k = \Pdual_{x_0,x_1}^k\otimes_\kk \sE^0_M$. It has a natural flat connection $\nabla$.

\begin{proposition}
There is a spectral sequence converging to $H^{j+k}_\dR(P_{x_0,x_1} M;\kk)$ with $E_1^{j,k} = E^j(M;\sP_{x_0,x_1}^k)$ and $d_1 = \nabla$, so that $E_2^{j,k} = H^j(M;\sP^k_{x_0,x_1})$.
\end{proposition}

\begin{proof}
The proof is similar to the proof of Proposition~\ref{prop:ss}. As there, we choose a quasi-isomorphic sub DGA $A^\bdot$ of $E^\bdot(M)$ with $A^0 = \kk$. Since the restriction of the augmentation $\e_x$ to $A^\bdot$ does not depend on $x\in M$, for all $x,y\in M$, there is a natural quasi-isomorphism
$$
B(\kk,A^\bdot,\kk) \hookrightarrow \Ch^\bdot(P_{x,y}M).
$$

The de~Rham cohomology of the total space $P_{x_0,x_1}M$ is computed by the complex
\begin{equation}
\label{eqn:big-complex}
B(\kk,A^\bdot,E^\bdot(M))\otimes_{E^\bdot(M)} B(E^\bdot(M),A^\bdot,\kk).
\end{equation}
As a graded vector space, it is isomorphic to
$$
B(\kk,A^\bdot,\kk) \otimes E^\bdot(M) \otimes B(\kk,A^\bdot,\kk).
$$

Let $\sB^\bdot$ be the sheaf of DGAs over $M$ associated to the presheaf
$$
U \mapsto B(\kk,A^\bdot,E^\bdot(U))\otimes_{E^\bdot(U)} B(E^\bdot(U),A^\bdot,\kk).
$$
It is isomorphic to $B(\kk,A^\bdot,\kk) \otimes \sE_M^\bdot \otimes B(\kk,A^\bdot,\kk)$. Corollaries~\ref{cor:stalk} and \ref{cor:path-dr} imply that the homology sheaf associated to $\sB^\bdot$ is locally constant and isomorphic to $\Pdual_{x_0,x_1}\otimes_\kk \sE^0_M$. It is also isomorphic to
$$
H^0(B(\kk,A^\bdot,\kk))^{\otimes 2}\otimes \sE^0_M.
$$

The complex (\ref{eqn:big-complex}) can be filtered by degree in $E^\bdot(M)$. The resulting spectral sequence converges to $H^\bdot_\dR(P_{x_0,x_1}M;\kk)$. The $E_1$ term of the spectral sequence is given by
$$
E_1^{j,k} \cong E^j(M,\sP_{x_0,x_1}^k).
$$
We have to show that under this isomorphism, $d_1$ corresponds to the differential $\nabla$ given by the canonical flat connection on $\sP_{x_0,x_1}$.

To do this, first note that the $E_1$ term is also isomorphic to
$$
H^\bdot(B(\kk;A^\bdot;\kk))^{\otimes 2}\otimes \sE^\bdot_N.
$$
We need only show that $\nabla$ and $d_1$ agree on sections in $H^\bdot(B(\kk;A^\bdot;\kk))^{\otimes 2}\otimes \kk$.

In the case $k=0$, the equality of $d_1$ and $\nabla$ reduces to the Fundamental Theorem of Calculus. Fix $\gamma \in P_{x_0,x_1}M$ and $w_1,\dots,w_r\in A^1$. The Fundamental Theorem of Calculus implies that the the value of the exterior derivative of the function
$$
t \mapsto \int_{\gamma_t'} (w_1|\dots|w_j) \int_{\gamma_t''}(w_{j+1}|\dots|w_r)
$$
at $t=1$ is
\begin{multline*}
\langle w_j,\dot{\gamma}_t'(1)\rangle\int_{\gamma_t'}(w_1|\dots|w_{j-1})\int_{\gamma_t''}(w_{j+1}|\dots|w_r)
\cr
-\langle w_{r+1},\dot{\gamma}_t''(0)\rangle\int_{\gamma_t'} (w_1|\dots|w_r) \int_{\gamma_t''}(w_{r+2}|\dots|w_r)
\end{multline*}
This corresponds to the degree $(1,0)$ differential (and therefore the $d_1$ differential)
\begin{multline*}
[w_1|\dots|w_j] \otimes 1 \otimes [w_{j+1}|\dots|w_r]
\cr
\mapsto (-1)^{\deg[w_1|\dots|w_{r-1}]}[w_1|\dots|w_{j-1}]\otimes w_j \otimes [w_{j+1}|\dots|w_r] \cr
- (-1)^{\deg[w_1|\dots|w_{r}]}[w_1|\dots|w_j] \otimes w_{j+1} \otimes [w_{j+2}|\dots|w_r]
\end{multline*}
of the double complex (\ref{eqn:big-complex}).
\end{proof}

\begin{corollary}
\label{cor:rat-Kpi1}
If $M$ is a rational $K(\pi,1)$, then for all $j\ge 0$
$$
H^j(M;\sP_{x_0,x_1}) \cong H^j_\dR(P_{x_0,x_1}M;\kk) \cong \Hdual^j(P_{x_0,x_1}M;\kk)
$$
and these vanishes when $j>0$.
\end{corollary}

\subsection{Relative de~Rham cohomology}

We continue using the notation of the previous section. Before giving the factorization of the continuous dual of the KK-action, we need to understand how to express the relative cohomology group $H^\bdot(M,\{x_0,x_1\};\sP_{x_0,x_1})$ in terms of iterated integrals. For simplicity, we suppose that $M$ is a rational $K(\pi,1)$. By Corollary~\ref{cor:rat-Kpi1}, there is an isomorphism
$$
H^j(M;\sP_{x_0,x_1}) \cong H^j_\dR(P_{x_0,x_1} M;\kk)
$$
and this vanishes for all $j>0$. This vanishing implies that the sequence
\begin{multline}
\label{eqn:ses_rel_coho}
0 \to H^0(M;\sP_{x_0,x_1}) \to H^0(\{x_0,x_1\};\sP_{x_0,x_1})
\cr \overset{\delta}{\to} H^1(M,\{x_0,x_1\};\sP_{x_0,x_1}) \to H^1(M;\sP_{x_0,x_1})
\end{multline}
is exact.

To declutter the notation, we temporarily denote $B(\kk,A^\bdot,\kk)$ by $B(A^\bdot)$. There is a natural inclusion $B(A^\bdot) \hookrightarrow \Ch^\bdot(P_{x_0,x_1}M)$. Since $M$ is a rational $K(\pi,1)$, with this notation, $H^\bdot(M,\sP_{x_0,x_1})$ can be computed as the cohomology of the complex $B(A^\bdot)\otimes E^\bdot(M)\otimes B(A^\bdot)$. The cohomology of the fibers of $p : P_{x_0,x_1}M \to M$ over $x_0$ and $x_1$ are both computed by the complex
$$
B(A^\bdot)\otimes \kk \otimes B(A^\bdot).
$$
The restrictions to the fibers over $x_0$ and $x_1$
$$
\e_0,\ \e_1 : B(A^\bdot)\otimes E^\bdot(M)\otimes B(A^\bdot) \to B(A^\bdot)\otimes \kk \otimes B(A^\bdot)
$$
are induced by the augmentations
$$
\e_{x_0},\ \e_{x_1} : E^\bdot(M) \to \kk,
$$
on the middle factor. The relative cohomology $H^\bdot(M,\{x_0,x_1\};\sP_{x_0,x_1})$ is thus computed by the complex
$$
\cone\big[B(A^\bdot)\otimes E^\bdot(M) \otimes B(A^\bdot) \overset{\e_0 \oplus \e_1}{\To} B(A^\bdot)\otimes \kk \otimes B(A^\bdot) \oplus B(A^\bdot)\otimes \kk \otimes B(A^\bdot)\big][-1].
$$
Denote it by $B^\bdot(M,\{x_0,x_1\};\sP_{x_0,x_1})$. An element of this complex of degree $j$ will be written in the form
$$
\begin{bmatrix}
J & & K \cr & I &
\end{bmatrix}
\quad
\text{where}
\quad
\begin{cases}
I \in [B(A^\bdot)\otimes E^\bdot(M) \otimes B(A^\bdot)]^j\cr
J,\ K \in [B(A^\bdot)\otimes \kk \otimes B(A^\bdot)]^{j-1}.
\end{cases}
$$
Its differential is given by
$$
d \begin{bmatrix}
J & & K \cr & I &
\end{bmatrix}
=
\begin{bmatrix}
\xymatrix@C=-4pt@R=8pt{
\e_0(I) - dJ & & \e_1(I) - dK \cr & dI &
}
\end{bmatrix}.
$$

The exactness of (\ref{eqn:ses_rel_coho}) implies that the image of the connecting homomorphism $\delta$ in (\ref{eqn:ses_rel_coho}) is a pushout. Since $H^1(M;\sP_{x_0,x_1})$ vanishes when $M$ is a rational $K(\pi,1)$, we have:

\begin{lemma}
If $M$ is a rational $K(\pi,1)$, the diagram
$$
\xymatrix{
H^0_\dR(P_{x_0,x_1}M) \ar[r]^(.4){\Delta_0} \ar[d]_{\Delta_1} &
H^0_\dR(\Lambda_{x_0}M) \otimes H^0_\dR(P_{x_0,x_1}M) \ar[d]^{\iota_0}
\cr
H^0_\dR(P_{x_0,x_1}M)\otimes H^0_\dR(\Lambda_{x_0}M) \ar[r]_{\iota_1} &
H^1(M,\{x_0,x_1\};\sP_{x_0,x_1})
}
$$
is a pushout, where $\Delta_0$ and $\Delta_1$ are induced by the multiplication maps
$$
\Lambda_{x_0}M\times P_{x_0,x_1}M \to P_{x_0,x_1}M \text{ and }P_{x_0,x_1}M \times \Lambda_{x_1}M \to P_{x_0,x_1}M
$$
respectively and $\iota_0$ and $\iota_1$ are induced by the maps
$$
\iota_\alpha : B(A^\bdot)\otimes B(A^\bdot) \to B(M,\{x_0,x_1\};\sP_{x_0,x_1})
$$
defined by
$$
\iota_0(I\otimes J) = 
\begin{bmatrix}
\xymatrix@C=-42pt@R=8pt{
I\otimes 1 \otimes J & & 0 \cr & 0
}
\end{bmatrix}
\text{ and }
\iota_1(I \otimes J) = 
\begin{bmatrix}
\xymatrix@C=-42pt@R=8pt{
0 & & I\otimes 1 \otimes J \cr & 0
}
\end{bmatrix}.
$$
\end{lemma}

It will be useful to note that if $I\in B(A^\bdot)$ represents an element of $H^0_\dR(P_{x_0,x_1})M$ and $\Delta I = \sum_\alpha I_\alpha'\otimes I_\alpha''$, where $I_\alpha', I_\alpha'' \in H^0(B(A^\bdot))$, then
$$
d\begin{bmatrix}
\xymatrix@C=-16pt@R=8pt{
0 & & 0 \cr & \sum_\alpha I_\alpha'\otimes 1 \otimes I_\alpha''
}
\end{bmatrix}
=
\begin{bmatrix}
\xymatrix@C=-42pt@R=8pt{
\iota_0 \Delta I && \iota_1 \Delta I \cr & 0
}
\end{bmatrix}.
$$

\subsection{A formula for the dual of $\csp$}

Suppose that $M$ is a connected smooth manifold with finite Betti numbers. As in the previous section, we assume for simplicity that $M$ is a rational $K(\pi,1)$. Recall the definition of $\csp$ from Section~\ref{sec:CS}.

\begin{proposition}
The map
$$
\csp^t : H^1(M,\{x_0,x_1\} ;\sP_{M;x_0,x_1}) \to H^0_\dR(P_{x_0,x_1}M)
$$
is induced by the map
$$
\csp^\ast : B^\bdot(M,\{x_0,x_1\};\sP_{x_0,x_1}) \to B(\kk,A^\bdot,\kk)
\hookrightarrow \Ch^\bdot(P_{x_0,x_1}M)
$$
defined by
\begin{multline}
\label{eqn:formula_bkk}
\begin{bmatrix}
\xymatrix@C=-42pt@R=8pt{
I_0 \otimes 1 \otimes J_0 & & I_1 \otimes 1 \otimes J_1 \cr
& [w_1|\dots|w_j]\otimes w \otimes [w_{j+1}|\dots|w_r]
}
\end{bmatrix}
\cr
\mapsto
[w_1|\dots |w_j|w|w_{j+1}|\dots|w_r] + \e(I_0)J_0 - \e(J_1)I_1,
\end{multline}
where $w,w_1,\dots,w_r \in E^1(M)$ and $I_0,I_1,J_0,J_1 \in B(\kk,A^\bdot,\kk)^0$.
\end{proposition}

In the expression above, the augmentation $I_0 \mapsto \e(I_0)$ corresponds to evaluation on trivial loop in $\Lambda_{x_0}M$ and $J_1 \mapsto \e(J_1)$ to evaluation on the trivial loop in $\Lambda_{x_1}M$. So the maps $I_0\otimes J_0\mapsto \e(I_0)J_0$ and $I_1\otimes J_1 \mapsto \e(J_1)I_1$ are induced by the projections $p^{-1}(x_0) \to P_{x_0,x_1}M$ and $p^{-1}(x_1) \to P_{x_0,x_1}M$, respectively.

\begin{proof}
Since
$$
E^j(M;\sP_{x_0,x_1}) = H^0(B(A^\bdot))\otimes E^j(M) \otimes H^0(B(A^\bdot)) \subseteq B(A^\bdot)\otimes E^\bdot(M) \otimes B(A^\bdot),
$$
every element of $E^j(M;\sP_{x_0,x_1})$ is of the form
$$
W = \sum I_k \otimes w_k \otimes J_k,
$$
where $w_k \in E^k(M)$ and $I_k,J_k\in H^0(B(A^\bdot)) \subseteq B(A^\bdot)$. Its exterior derivative is
\begin{align*}
\nabla W &= \sum_k \big(\nabla_0 I_k) \wedge w_k \otimes J_k + I_k \otimes dw_k \otimes J_k + I_k \otimes w_k \wedge \nabla_1 J_k
\cr
&=
\sum_k \big(
I_{k,\alpha} \otimes (w_{k,\alpha} \wedge w_k) \otimes J_k 
+ I_k \otimes dw_k \otimes J_k
+  I_k \otimes (w_k\wedge \xi_{k,\beta}) \otimes J_{k,\beta}\big),
\end{align*}
where
$$
\nabla_0 I_k = \sum_\alpha I_{k,\alpha}\otimes w_{k,\alpha}
\in H^0(B(A^\bdot))\otimes A^\bdot \subseteq B(\kk,A^\bdot,A^\bdot) \subseteq \Ch^1(P_{x_0,M}M)
$$
and
$$
\nabla_1 J_k = -\sum_\beta  \xi_{k,\beta} \otimes J_{k,\beta}
\in A^\bdot \otimes H^0(B(A^\bdot)) \subseteq B(A^\bdot,A^\bdot,\kk) \subseteq \Ch^1(P_{M,x_1}M).
$$
Here $P_{M,x_1}M$ denotes the restriction of $PM \to M^2$ to $M\times\{x_1\}$, and $P_{x_0,M}M$ its restriction to $\{x_0\}\times M$.

The first task is to show that the formula vanishes on 1-coboundaries. The elements of degree 0 in $B^\bdot(M,\{x_0,x_1\};\sP_{x_0,x_1})$ are linear combinations of elements of the form
$$
\begin{bmatrix}
\xymatrix@C=-16pt@R=8pt{
0 && 0 \cr & I \otimes f \otimes J
}
\end{bmatrix},
$$
where $f \in E^0(M)$ and $I,J \in H^0(B(A^\bdot))$. Write
$$
\nabla_0 I = \sum_\alpha I_\alpha \otimes w_\alpha
\text{ and }
\nabla_1 J = -\sum_\beta \xi_\beta \otimes J_\beta.
$$
Then
$$
\nabla(I \otimes f \otimes J) = I\otimes df \otimes J + \sum I_\alpha\otimes f w_\alpha \otimes J - \sum I \otimes f\xi_\beta \otimes J_\beta.
$$
At this stage, it is useful to introduce the notation $[I|w]$ for the concatenation of $I$ and $w\in A^+$, etc. Then the expressions above imply that
$$
I = \sum_\alpha [I_\alpha|w_\alpha] \text{ and } J = \sum_\beta [\xi_\beta|J_\beta]
$$
and that
$$
\csp^\ast : 
\begin{bmatrix}
\xymatrix@C=-16pt@R=8pt{
0 && 0 \cr & I'\otimes w \otimes J' 
}
\end{bmatrix}
\mapsto [I'|w|J'].
$$

Since
$$
d \begin{bmatrix}
\xymatrix@C=-16pt@R=8pt{
0 && 0 \cr & I \otimes f \otimes J
}
\end{bmatrix}
=
\begin{bmatrix}
\xymatrix@C=-16pt@R=8pt{
f(x_0)\iota_0(I\otimes J) && f(x_1)\iota_1(I\otimes J) \cr & \nabla(I \otimes f \otimes J)
}
\end{bmatrix}
$$
the relations (\ref{eqn:relns}) imply that this lies in the kernel of the map (\ref{eqn:formula_bkk}).\footnote{The easiest way to verify this is to note that $B(A^\bdot)$ can be decomposed $B(A^\bdot) = \kk[\blank] \oplus \ker \e$, where $\ker \e$ is the span of the $[w_1|\dots|w_r]$ with $r>0$. One then treats the cases (1) $I,J\in \ker\e$, (2) $I\in \ker \e$ and $J=[\blank]$, (3) $I = [\blank]$ and $J \in \ker \e$, (4) $I=J=[\blank]$. These correspond to the 4 types of relations.}

The next task is to show that (\ref{eqn:formula_bkk}) takes closed elements to closed iterated integrals. This follows from the fact that if $I,J\in H^0(B(A^\bdot))$, as above, and if $w\in E^1(M)$, then
$$
\nabla I\otimes w \otimes J = \sum_\alpha I_\alpha \otimes (w_\alpha\wedge w)\otimes J + I \otimes dw \otimes J + \sum_\beta I \otimes (w\wedge w_\beta)\otimes J_\beta,
$$
so that
$$
d \begin{bmatrix}
\xymatrix@C=-16pt@R=8pt{
0 && 0 \cr & I \otimes w \otimes J
}
\end{bmatrix}
=
\begin{bmatrix}
\xymatrix@C=-16pt@R=8pt{
0 && 0 \cr & \nabla(I \otimes w \otimes J)
}
\end{bmatrix}
$$
Then the formula (\ref{eqn:diff}) for the differential in the $B(A^\bdot)$ implies that this goes to $- d(I|w|J)$ under (\ref{eqn:formula_bkk}).

The final task is to verify that the formula (\ref{eqn:formula_bkk}) induces a map dual to $\csp$. To do this, it suffices to verify it on an element
$$
\begin{bmatrix}0 & & 0 \cr & \varphi \end{bmatrix}
$$
of $B(M,\{x_0,x_1\};\sP_{x_0,x_1})$, where
$$
\varphi := [w_1|\dots|w_j]\otimes w \otimes [w_{j+1}|\dots|w_r]
$$
as every element of $H^1(M,\{x_0,x_1\};\sP_{x_0,x_1})$ is represented by a linear combination of such elements. Since $\csp(\gamma) = [\gammahat]$ and since
$$
\langle \textstyle{\begin{bmatrix}0 & & 0 \cr & \varphi \end{bmatrix}},\csp(\gamma)\rangle
= \langle \varphi,\gammahat \rangle,
$$
it suffices to show that 
$$
\langle \varphi, \gammahat \rangle = \int_\gamma (w_1|\dots|w_j|w|w_{j+1}|\dots|w_r)
$$
for each $\gamma \in P_{x_0,x_1}M$. To prove this, fix $\gamma$ and let $f$ and $f_j$ be the smooth functions on the unit interval defined by
$$
\gamma^\ast w = f(t) dt \text{ and } \gamma^\ast w_j = f_j(t)dt,\qquad j=1,\dots, r.
$$
Then the value of $\varphi$ on $\gammahat$ is
\begin{align*}
\langle \varphi, \gammahat \rangle
&= \int_0^1 \langle \int (w_1|\dots|w_j),\gamma_t' \rangle\,
\langle \int(w_{j+1}|\dots|w_r),\gamma_t''\rangle \, f(t)dt
\cr
&= \int_0^1 f(t)
\idotsint\limits_{0\le t_1 \le \dots \le t_j \le t} f_1(t_1)\dots f_j(t_j)
\idotsint\limits_{t\le t_{j+1} \le \dots \le t_r \le 1} f_{j+1} \dots f_r(t_r)
\cr
&=
\idotsint\limits_{0 \le t_1 \le \dots \le t_j \le t \le t_{j+1} \le \dots t_r \le 1}
f_1(t_1) \dots f_j(t_j) f(t) f_{j+1}(t_{j+1}) \dots f_r(t_r)
\cr
&= \int_\gamma (w_1|\dots|w_j|w|w_{j+1}|\dots|w_r).
\end{align*}
\end{proof}

When $M$ is a compact manifold with boundary and a rational $K(\pi,1)$, we define $\cspdual$ to be the composite
$$
\xymatrix{
H^1(M,\{x_0,x_1\}\cup \partial M;\Pdual_M) \ar[r] & H^1(M,\{x_0,x_1\};\Pdual_M) \ar[r]^(.55){\csp^t} & \Hdual^0(P_{x_0,x_1}M).
}
$$
Here we are using the de~Rham theorems to make the identifications
$$
H^1(M,\{x_0,x_1\};\Pdual_M) \cong H^1_\dR(M,\{x_0,x_1\};\sP_M),\
\Hdual^0(P_{x_0,x_1}M) \cong H^0_\dR(P_{x_0,x_1}M).
$$

\subsection{The continuous dual of the KK-action}

Suppose that $\Xbar$ is a compact, connected, oriented surface and that $Y$ is the union of $\partial X$ with a finite set. In addition, suppose that $x_0,x_1\in \Xbar$. These points may (or may not) lie in $Y$. Set
$$
X= \Xbar- Y \text{ and } X' = X \setminus \{x_0,x_1\}.
$$
Denote the inclusion by $j : X'\to X$.

The proof of the following result is similar to that of Lemma~\ref{lem:cup_dual} and is omitted. Since $X$ is a $K(\pi,1)$, $\Pdual^j_{X;x_0,x_1} = 0$ when $j>0$. As in Section~\ref{sec:pullback_fibrations}, we will denote $\Pdual_{X;x_0,x_1}^0$ by $\Pdual_{X;x_0,x_1}$.

\begin{lemma}
There is a map
\begin{multline}
\label{eqn:ext_cup_dual}
\smile^\ast\ : H_2(\Xbar,Y\cup\{x_0,x_1\};\Pdual_{X;x_0,x_1})
\cr
\to H_1(\Xbar,Y\cup\{x_0,x_1\};\Ldual_{X'})\otimes H_1(X',j^\ast\Pdual_{X;x_0,x_1})
\end{multline}
that is the continuous dual of the map
\begin{multline*}
\smile\ : 
H^1(\Xbar,Y\cup\{x_0,x_1\};\bL_{X'})\otimes H^1(X',j^\ast\bP_{X;x_0,x_1})
\cr
\to
H^2(\Xbar,Y\cup\{x_0,x_1\};\bP_{X;x_0,x_1})
\end{multline*}
induced by the cup product and the fiberwise multiplication map
$\bL_{X'}\otimes j^\ast\bP_{X;x_0,x_1} \to j^\ast\bP_{X;x_0,x_1}$ via the natural pairings
\begin{equation}
\label{eqn:dual_pairings}
\bL_{X'} \otimes \Ldual_{X'} \to \kk_{X} \text{ and }
\bP_{X;x_0,x_1} \otimes \Pdual_{X;x_0,x_1} \to \kk_X.
\end{equation}
\end{lemma}

\begin{proposition}
\label{prop:dual_kk}
The pairings (\ref{eqn:dual_pairings}) induce a pairing of the composition of the arrows on the bottom and two sides of
$$
\xymatrix@C=12pt{
\Hdual^0(P_{x_0,x_1} X) \ar[d]^\cong \ar[r]^(.45){\kappa^\ast} &
\Hdual^0(\Lambda X') \otimes \Hdual^0(P_{x_0,x_1} X)
\cr
H^0(X';\Pdual_{X;x_0,x_1}) \ar[d]^\cong &
H^1(X';\Ldual_{X'})\otimes H^1(\Xbar,Y \cup \{x_0,x_1\};\Pdual_{X;x_0,x_1}) 
\ar[u]_{\csdual\otimes \cspdual}
\cr
H_2(\Xbar,Y\cup\{x_0,x_1\};\Pdual_{X;x_0,x_1}) \ar[r]^(.4){\smile^\ast} &
H_1(\Xbar,Y\cup\{x_0,x_1\};\Ldual_{X'})\otimes H_1(X',j^\ast\Pdual_{X;x_0,x_1})\ar[u]_\cong
}
$$
pairs with the factorization (\ref{eqn:action}) of the KK-action. Consequently, the dual of the KK-action is continuous.
\end{proposition}

The continuity of the KK-action was proved directly by Kawazumi and Kuno in \cite[\S 4.1]{kk:groupoid}.

\section{Hodge Theory Background}
\label{sec:Hodge}

The reader is assumed to have some familiarity with Deligne's mixed Hodge theory. The basic reference is \cite{deligne:hodge2}. The book \cite{peters-steenbrink} is a comprehensive introduction. The paper \cite{hain:bowdoin} is an introduction to the Hodge theory of the unipotent fundamental groups of smooth complex algebraic varieties. In this section, we summarize the results from Hodge theory needed in the proofs of the main results.

\subsection{Basics}

Recall that a {\em mixed Hodge structure} (hereafter, a MHS) $V$ consists of a finite dimensional $\Q$-vector space $V_\Q$, an increasing {\em weight filtration}
$$
0 = W_n V_\Q \subseteq \dots \subseteq W_{j-1}V_\Q \subseteq W_j V_\Q \subseteq \dots \subseteq W_N V_\Q = V_\Q
$$
of $V_\Q$, and a decreasing {\em Hodge filtration}
$$
V_\C = F^a V_\C \supseteq \dots \supseteq F^p V_\C \supseteq F^{p+1} V_\C \supseteq \dots \supseteq F^b V_\C = 0,
$$
where $V_\C$ denotes $V_\Q\otimes\C$. These are required to satisfy the condition that, for all $m\in\Z$, the $m$th weight graded quotient $\Gr^W_m V$ of $V$, whose underlying $\Q$ vector space is
$$
\Gr^W_m V := W_m V_\Q/W_{m-1}V_\Q
$$
is a Hodge structure of weight $m$. This means simply that
$$
\Gr^W_m V_\C = \bigoplus_{p+q=m} (F^p \Gr^W_m V_\C) \cap (\Fbar^q \Gr^W_m V_\C)
$$
where $\Fbar^q V_\C$ is the conjugate of $F^q V_\C$ under the action of complex conjugation on $V_\C$. A MHS $V$ is said to be {\em pure} of weight $m\in \Z$ if $\Gr^W_r V = 0$ when $r\neq m$. A {\em Hodge structure} of weight $m$ is simply a MHS that is pure of weight $m$.

Morphisms of MHS are defined in the obvious way --- they are weight filtration preserving $\Q$-linear maps of the underlying $\Q$-vector spaces which induce Hodge filtration preserving maps after tensoring with $\C$. The category of MHS is a $\Q$-linear abelian tensor category. This is not obvious, as the category of filtered vector spaces is not an abelian category. For each $m\in \Z$, the functor $\Gr^W_m$ from the category of MHS to $\Vec_\Q$, the category of $\Q$ vector spaces, is exact. In particular, this implies that if $\phi : V \to V'$ is a morphism of MHS, then there are natural isomorphisms
$$
\Gr^W_\bdot \ker \phi \cong \ker \Gr^W_\bdot \phi
\text{ and }
Gr^W_\bdot \im \phi \cong \im \Gr^W_\bdot \phi.
$$

Most invariants of complex algebraic varieties that one can compute using differential forms carry a natural mixed Hodge structure; morphisms of varieties induce morphisms of MHS. These MHS have the additional property that they are {\em graded polarizable}, which means that each weight graded quotient admits a polarization, which is an inner product that satisfies the Riemann--Hodge bilinear relations. All MHS of ``geometric origin'' are graded polarizable. Graded polarizations play an important, if hidden, role in the theory as we'll indicate below.

The Hodge structure $\Q(1)$ is the 1-dimensional Hodge structure $V$ of weight $-2$ with Hodge filtration
$$
V_\C = F^{-1} V_\C \supset F^0 V_\C = 0.
$$
It arises naturally as $H_1(\C^\ast;\Q)$. Its dual is the Hodge structure $\Q(-1) \cong H^1(\C^\ast;\Q)$ which has weight 2. The Hodge structure $\Q(n)$ is defined to be $\Q(1)^{\otimes n}$ when $n\ge 0$ and $\Q(-1)^{\otimes(-n)}$ when $n<0$. If $X$ is an irreducible projective variety of dimension $d$, then there are natural isomorphisms
$$
H^{2d}(X) \cong \Q(-d) \text{ and } H_{2d}(X) \cong \Q(d).
$$
The {\em Tate twist} $V(n)$ of a MHS $V$ is defined by
$$
V(n) := V\otimes \Q(n).
$$
Tensoring with $\Q(n)$ shifts the weight filtration by $-2n$ and the Hodge filtration by $-n$.

\subsection{Tannakian considerations}
\label{sec:tannaka}

The rational vector space $V_\Q$ that underlies a MHS $V$ is naturally, though not canonically, isomorphic to its associated weight graded $\Gr^W_\bdot V$. These splittings form a principal homogeneous space under the action of the unipotent radical of the Mumford--Tate group of $V$. This fact will allow us to use Hodge theory to give natural (but not canonical) isomorphisms of the Goldman Lie algebra with its associated graded Lie algebra and to show that they comprise a principal homogeneous space under a unipotent $\Q$-group. The best way to explain this is to use tannakian formalism. An excellent reference for Tannakian categories is Deligne's paper \cite{deligne:tannakian}.

Denote the category of {\em graded polarizable} MHS by $\MHS$. The semi-simple objects of this category are direct sums of Hodge structures.\footnote{The existence of a polarization on a pure Hodge structure implies that it is a semi-simple object of $\MHS$.} The functor
$$
\omega : \MHS \to \Vec_\Q
$$
that takes a MHS to its underlying vector space is faithful and preserves tensor products. This implies that $\MHS$ is a $\Q$-linear neutral tannakian category and is therefore equivalent to the category of representations of the affine $\Q$-group
$$
\pi_1(\MHS,\omega) := \Aut^{\otimes} \omega.
$$
The $\Q$-vector space underlying each MHS is naturally a $\pi_1(\MHS,\omega)$-module.

\begin{definition}
The {\em Mumford--Tate group} $\MT_V$ of a MHS $V$ is the image of the associated homomorphism
$$
\rho_V : \pi_1(\MHS,\omega) \to \Aut_\Q V_\Q.
$$
It is an affine algebraic $\Q$-group.
\end{definition}

Denote the sub-category of semi-simple objects of $\MHS$ by $\MHS^\ss$. Since its objects are direct sums of polarizable Hodge structures, $\MHS^\ss$ is a tannakian subcategory of $\MHS$. The inclusion $\MHS^\ss \hookrightarrow \MHS$ is fully-faithful and therefore induces a surjection $\pi_1(\MHS,\omega) \to \pi_1(\MHS^\ss,\omega)$. Its kernel is prounipotent, so that one has an extension
\begin{equation}
\label{eqn:pi_1MHS}
1 \to \U^\MHS \to \pi_1(\MHS,\omega) \to \pi_1(\MHS^\ss,\omega) \to 1.
\end{equation}

The category of graded vector spaces $\Vec^\bdot_\Q$ is tannakian with fiber functor $\oplus : V^\bdot \mapsto \oplus_m V^m$. There is a natural isomorphism $\pi_1(\Vec^\bdot_\Q,\oplus) \cong \Gm$,
where $\Gm$ acts on $V^m$ by the $m$th power of the standard character. Since every object of $\MHS^\ss$ is graded by weight, the fiber functor $\omega$ of $\MHS^\ss$ factors
$$
\MHS^\ss \To \Vec^\bdot_\Q \overset{\oplus}{\To} \Vec_\Q.
$$
Consequently, $\MHS^\ss \to \Vec_\Q^\bdot$ induces a central cocharacter
$$
\chi : \Gm \to \pi_1(\MHS^\ss,\omega)
$$
of $\Q$-groups.

Levi's Theorem implies that there is a lift $\chitilde : \Gm \to \pi_1(\MHS,\omega)$ (no longer central) and that any two such lifts are conjugate by an element of $\U^\MHS(\Q)$.

The abelianization $H_1(\U^\MHS)$ is an inverse limit of $\pi_1(\MHS^\ss,\omega)$-modules and is thus a pro-object of $\MHS^\ss$. The fact that $\Ext^1_\MHS(\Q(0),V)$ vanishes whenever $V$ is a polarizable Hodge structure of weight $\ge 0$ implies that $\Gr^W_m H_1(\U^\MHS)=0$ for all $m\ge 0$. This means that (\ref{eqn:pi_1MHS}) is a negatively weighted extension in the sense of \cite{hain-matsumoto}. The following result follows from \cite[\S 3]{hain-matsumoto}.

\begin{proposition}
\label{prop:splittings}
Each lift $\chitilde$ of $\chi$ determines an isomorphism
$$
V_\Q \overset{\sigma^\chitilde_V}{\To} \bigoplus_m \Gr^W_m V_\Q
$$
which is natural in the sense that if $\phi : V \to V'$ is a morphism of MHS, then the diagram
$$
\xymatrix{
V_\Q \ar[r]^(.35){\sigma^\chitilde_V}\ar[d]_{\phi} & \bigoplus_m \Gr^W_m V_\Q \ar[d]^{\Gr^W\phi} \cr
V_\Q' \ar[r]^(.35){\sigma^\chitilde_{V'}} & \bigoplus_m \Gr^W_m V_\Q'
}
$$
commutes for all $\phi$. The isomorphisms $\sigma^\chitilde_V$ are also compatible with tensor products and taking duals.
\end{proposition}

\begin{remark}
\label{rem:splittings}
For each object $V$ of $\MHS$, the set of splittings $\sigma^\chitilde_V$ is a principal homogeneous space over the unipotent radical $U_V^\MT$ of $\MT_V$.
\end{remark}

\subsection{Admissible variations of MHS}
\label{sec:vmhs}

The local systems that occur naturally in Hodge theory are admissible variations of MHS. Very roughly speaking, a variation $\bV$ of MHS over a smooth complex algebraic variety $X$ consists of a $\Q$-local system $\bV_\Q$ over $X$, an increasing filtration $W_\bdot$ of it by sub-local systems, and a decreasing filtration $F^\bdot$ of the associated flat holomorphic vector bundle $\sV := \bV\otimes_\Q \cO_X$ by holomorphic sub-bundles. The first requirement is that each fiber of $\bV$ is a MHS with the induced Hodge and weight filtrations. But to be a variation of MHS, many other conditions need to be satisfied. Here we briefly mention them. Basic references include the foundational papers \cite{steenbrink-zucker} and \cite{kashiwara} of Steenbrink--Zucker, and the paper \cite{hain-zucker} on unipotent variations of MHS, which is particularly relevant in this paper.

To enumerate the extra conditions, we need to write $X = \Xbar - D$, where $\Xbar$ is smooth and complete, and where $D$ is a normal crossings divisor in $\Xbar$. This is possible by Hironaka's resolution of singularities. Denote Deligne's canonical extension of $\sV$ to $\Xbar$ by $\sVbar$. For simplicity, we suppose that the monodromy of $\bV$ about each component of $D$ is unipotent. This condition is satisfied by all variations of MHS in this paper and implies that canonical extension commutes with tensor products. In order that $\bV$ be an admissible variation of MHS, the following additional conditions must be satisfied:
\begin{enumerate}

\item \label{item:hodge} The Hodge bundles $F^p\sV$ extend to sub-bundles $F^p\sVbar$ of $\sVbar$;

\item for each $p$, the connection $\nabla$ satisfies ``Griffiths transversality''
$$
\nabla : F^p\sV \To \Omega^1_\Xbar(\log D)\otimes F^{p-1}\sVbar;
$$

\item variation is graded polarizable in the sense that there exist flat inner products on each $\Gr^W_m \bV$ which induce a polarization on each fiber;

\item for each stratum of $D$, there is a ``relative weight filtration''.

\end{enumerate}

The category of admissible variations of MHS over $X$ will be denoted $\MHS(X)$. It is tannakian.

\subsection{Poincar\'e duality}

The next result follows from Saito's work \cite{saito:mhm} in the general case. The unipotent case, which is all we will need, is a consequence of \cite[Prop.~8.6]{hain-zucker}.

\begin{theorem}[Saito]
\label{thm:PD}
Suppose that $X$ is a smooth projective variety and that $Y$ and $Z$ are proper closed subvarieties of $X$. If $\bV$ and $\bW$ are admissible variations of MHS over $X-Y$ and $X-Z$, respectively, then
\begin{enumerate}

\item the cohomology groups $H^\bdot(X-Y,Z\setminus Y;\bV)$ and $H^\bdot(X-Z,Y\setminus Z;\bW)$ have natural mixed Hodge structures;

\item the cup product
$$
H^j(X-Y,Z\setminus Y;\bV) \otimes H^k(X-Z,Y\setminus Z;\bW) \to H^{j+k}(X,Y\cup Z;\bV\otimes\bW)
$$
is a morphism of MHS.

\end{enumerate}
\end{theorem}

\begin{corollary}
If $\bV$ is an admissible variation of MHS over a smooth variety $X$ of dimension $d$, then Poincar\'e duality 
$$
H_{2d-j}(X-Y,Z\setminus Y;\bV) \overset{\simeq}{\To} H^j(X-Z,Y\setminus Z;\bV)(d)
$$
is an isomorphism of MHS.
\end{corollary}

\begin{proof}
Denote the dual variation by $\bV^\ast$. There is a non-degenerate pairing $\bV\otimes \bV^\ast \to \Q(0)$ of admissible variations of MHS. Poincar\'e duality and the proposition imply that the pairing
$$
H^{2d-j}(X-Y,Z\setminus Y;\bV^\ast) \otimes H^j(X-Z,Y\setminus Z;\bV) \to H^{2d}(X,Y\cup Z) \cong H^{2d}(X) \cong \Q(-d)
$$
is a non-singular pairing of MHS.
\end{proof}

\subsection{Mixed Hodge structures on cohomology of path torsors}
\label{sec:mhs_paths}

Since every complex algebraic variety has the homotopy type of a finite complex, we have a natural isomorphism
$$
\Hdual^0(P_{x_0,x_1}X;\kk) \cong H^0_\dR(P_{x_0,x_1}X;\kk)
$$
when $\kk = \R$ and $\C$.

\begin{theorem}
If $X$ is a smooth complex algebraic variety and $x_0,x_1 \in X$, then there is a canonical ind-MHS on $\Hdual^0(P_{x_0,x_1}X)$. It is natural in the triple $(X,x_0,x_1)$ and admits a graded polarization, so that it is an object of ind-$\MHS$. The product
$$
\Hdual^0(P_{x_0,x_1}X)^{\otimes 2} \To \Hdual^0(P_{x_0,x_1}X)
$$
and coproducts
$$
\Hdual^0(P_{x_0,x_2}X) \To \Hdual^0(P_{x_0,x_1}X) \otimes \Hdual^0(P_{x_1,x_2}X)
$$
are morphisms of ind-MHS.
\end{theorem}

The dual statement is that $H_0(P_{x_0,x_1}X;\Q)^\wedge$ is an object of pro-$\MHS$.

\begin{corollary}
\label{cor:lambda}
If $X$ is a smooth complex algebraic variety, then $\Hdual^0(\Lambda X)$ is an object of ind-$\MHS$. This MHS depends functorially on $X$ and, for each $x\in X$, the restriction mapping $\Hdual^0(\Lambda X) \to \Hdual^0(\Lambda_x X)$ is a morphism of MHS.
\end{corollary}

\begin{proof}
Since this fact is central to the paper, we sketch two proofs. The first is to note that for each $n\ge 0$ the sequence
$$
\big(\Q\pi_1(X,x)/I^n\big)^{\otimes 2} \to \Q\pi_1(X,x)/I^n \to H_0(\Lambda X;\Q)/I^n H_0(\Lambda X) \to 0
$$
is exact, where the first map takes $u\otimes v$ to $uv-vu$. Since $\Q\pi_1(X,x)/I^n$ is a ring in $\MHS$, the first map is a morphism of MHS. Since $\MHS$ is closed under quotients, $H_0(\Lambda X;\Q)/I^n H_0(\Lambda X)$ has an induced MHS. The result follows as
$$
\Hdual^0(\Lambda X;\Q) = \varinjlim_n \Hom_\Q(H_0(\Lambda X;\Q)/I^n H_0(\Lambda X),\Q).
$$
This MHS does not depnd on the choice of $x\in X$ as these MHS on $\Hdual^0(\Lambda X)$ constructed from the $x\in X$ form an admissible variation of MHS over $X$ with trivial monodromy; the admissibility follows as this variation is a quotient of the variation whose fiber over $x\in X$ is $\Q\pi_1(X,x)^\wedge$, which is admissible. See \cite{hain:dht2} or \cite{hain-zucker}. The theorem of the fixed part \cite[Thm.~4.1]{steenbrink-zucker} implies that it is a constant variation of MHS.

The result also follows by combining Theorem~\ref{thm:pullback_DR} with \cite[Thm.~3.2.1]{hain:dht} applied to the de~Rham mixed Hodge complex \cite[Prop.~5.6.2]{hain:dht} associated to a smooth variety.
\end{proof}

For later use we show that the Adams operations $\psi_n$ (Sec.~\ref{sec:adams}) are morphisms of MHS.

\begin{proposition}
If $X$ is smooth complex algebraic variety, then for each $n \ge 0$, the Adams operation
$$
\psi_n : H_0(\Lambda X;\Q)^\wedge \to H_0(\Lambda X;\Q)^\wedge
$$
is a morphism of MHS.
\end{proposition}

\begin{proof}
Choose a base point $x\in X$. Since $H_0(\Lambda X;\Q)^\wedge$ is a quotient of $\Q\pi_1(X,x)^\wedge$ as a MHS, it suffices to show that the endomorphism $\psi_n$ of $\Q\pi_1(X,x)^\wedge$ is a morphism of pro-MHS. There is a natural isomorphism
$$
\Q\pi_1\big(X^n,(x,\dots,x)\big)^\wedge \cong \big[\Q\pi_1(X,x)^\wedge\big]^{\otimes n}
$$
The diagonal $X \to X^n$ induces the $n$th power
$$
\Delta_n : \Q\pi_1\big(X^n,(x,\dots,x)\big)^\wedge \to \big[\Q\pi_1(X,x)^\wedge\big]^{\otimes n}.
$$
of the coproduct. It is a morphism of pro-MHS as the coproduct is a morphism. The Adams map $\psi_n$ on $\Q\pi_1(X,x)^\wedge$ is the composition
$$
\xymatrix{
\Q\pi_1(X,x)^\wedge \ar[r]^(.45){\Delta_n} &
\big[\Q\pi_1(X,x)^\wedge\big]^{\otimes n} \ar[r]^(.55){\text{mult}} & 
\Q\pi_1(X,x)^\wedge
}
$$
and is therefore a morphism of MHS.
\end{proof}

\subsection{Path variations of MHS}

Suppose that $X$ is a smooth complex algebraic variety. Denote by $\Pdual_{X\times X}$ the local system over $X\times X$ whose fiber over $(x_0,x_1)$ is $\Hdual^0(P_{x_0,x_1}X;\Q)$.

\begin{theorem}[{Hain--Zucker \cite{hain-zucker}}]
\label{thm:HZ}
The local system $\Pdual_{X\times X}$ underlies an object of ind-$\MHS(X\times X)$. Its restriction to the diagonal is $\Ldual_X$. It underlies an object of ind-$\MHS(X)$.
\end{theorem}

\begin{corollary}
For each $(x_0,x_1)\in X^2$, the local system $\Pdual_{X;x_0,x_1}$ underlies an object of ind-$\MHS(X)$.
\end{corollary}

\begin{proof}
Denote the restriction of $\Pdual_{X\times X}$ to $A\times B$ by $\Pdual_{A\times B}$. This is an admissible variation of MHS over $A\times B$ when $A$ and $B$ are subvarieties of $X$. In particular, the restrictions of $\Pdual_{X\times X}$ to $\{x_0\}\times X$ and $X\times\{x_1\}$ are objects of ind-$\MHS(X)$. The result follows as $\Pdual_{X;x_0,x_1} = \Pdual_{x_0\times X}\otimes \Pdual_{X\times x_1}$.
\end{proof}

\begin{corollary}
\label{cor:restriction}
The restriction mapping $\Hdual^0(\Lambda X) \to H^0(X,\Ldual_X)$ is a morphism of MHS.
\end{corollary}

\begin{proof}
Since $\Ldual_X$ is in ind-$\MHS(X)$, the Theorem of the Fixed part \cite[Thm.~4.1]{steenbrink-zucker} implies that for all $x\in X$, the restriction mapping $H^0(X,\Ldual_X) \hookrightarrow \Hdual^0(\Lambda_x X)$ is an inclusion of MHS. Since the image of the restriction mapping $\Hdual^0(\Lambda X)$ is contained in $H^0(X,\Ldual_X)$, the last assertion of Corollary~\ref{cor:lambda} implies the result.
\end{proof}

\begin{lemma}
If $j : U \hookrightarrow X-\{x_0,x_1\}$ is the inclusion of a Zariski open subset, then the continuous duals
$$
\mu^\ast : \Ldual_U \to \Ldual_U \otimes \Ldual_U \text{ and }
\mu^\ast : j^\ast : \Pdual_{X;x_0,x_1} \to \Ldual_U\otimes j^\ast\Pdual_{X;x_0,x_1}
$$
of the natural pairings (\ref{eqn:pairings}) are morphisms in ind-$\MHS(U)$.
\end{lemma}

\begin{proof}
We will prove that the second map is a morphism of admissible variations over $U$. The proof that first map is a morphism is similar, but simpler.

Since $\Pdual_{X;x_0,x_1}$ and $\Ldual_U$ are admissible ind-variations, it suffices to show that its restriction to each fiber is a morphism of ind-MHS. To do this, it suffices to prove that the pairing
$$
H_0(\Lambda_x U)^\wedge \otimes H_0(P_{x_0,x}X)^\wedge\otimes H_0(P_{x,x_1})^\wedge \to H_0(P_{x_0,x}X)^\wedge \otimes H_0(P_{x,x_1})^\wedge
$$
dual to this map is a morphism of pro-MHS for each $x\in U$. Since this map takes $\alpha \otimes \gamma'\otimes\gamma''$ to $\gamma'\alpha\otimes \gamma''$, it suffices to show that the multiplication map
$$
H_0(\Lambda_x U)^\wedge \otimes H_0(P_{x_0,x}X)^\wedge \to H_0(P_{x_0,x}X)^\wedge
$$
is a morphism. But this follows from the properties of the MHS on path torsors.
\end{proof}

Denote the diagonal in $X\times X$ by $\Delta$. We also need to consider the local system $\bLambda_{X\times X}'$ over $X\times X$ whose fiber over $(x_0,x_1)$ is $\Hdual^0\big(\Lambda (X -\{x_0,x_1\});\Q\big)$.

\begin{proposition}
The local system $\bLambda_{X\times X}'$ underlies an object of ind-$\MHS(X^2-\Delta)$.
\end{proposition}

\begin{proof}
Fix a base point $\xbar \in X$. Let
$$
U = X\times X -
\big(\Delta \cup (\{\xbar\}\times X) \cup (X \times \{\xbar\})\big).
$$
Consider the the local system $\bV$ over $U$ whose fiber over $(x_0,x_1)$ is
$$
\Q\pi_1(X-\{x_0,x_1\},\xbar)^\wedge.
$$
The main result of \cite{hain:dht2} (or of \cite{hain-zucker}) implies that it is an object of pro-$\MHS(U)$. The dual of the restriction of $\bLambda_{X\times X}'$ to $U$ is a quotient of $\bV$. The quotient map is a morphism of ind-MHS on each fiber. This implies that the restriction of $\bLambda_{X\times X}'$ to $U$ is in ind-$\MHS(U)$. But since it extends to a local system over $X\times X - \Delta$, it is also an object of ind-$\MHS(X^2-\Delta)$. Alternatively, one can show that it extends as a variation of MHS from $U$ to $X^2 - \Delta$ using the classification in \cite{hain-zucker}. Just use the fact that it has trivial monodromy about the divisors $\{\xbar\}\times X$ and $X \times \{\xbar\}$.
\end{proof}

\subsection{Tangent vectors and limit MHS}

Suppose that $\Xbar$ is a smooth complex projective variety and that $D$ is a divisor with normal crossings in $\Xbar$. Set $X = \Xbar - D$ and suppose that $\bV$ is an admissible variation of MHS over $X$. For simplicity, we suppose that the local monodromy of $\bV$ about each codimension 1 component of $D$ is unipotent. This condition holds in all variations used in this paper. Then, for each $p\in D$ and tangent vector $\vv \in T_p \Xbar$ not tangent to any component of $D$ at $p$, there is a {\em limit} MHS $V_\vv$. The $\Q$-vector space underlying $V_\vv$ can be thought of as the fiber of $\bV$ over $\vv$. We will explain this in the current context in more detail in Section~\ref{sec:tangent_vecs}.

Here is a brief explanation of the limit MHS. The complex vector space underlying $V_\vv$ is the fiber $V_p$ of Deligne's canonical extension $\sV$ to $\Xbar$ of the flat vector bundle $\bV\otimes_\Q \cO_X$. Its Hodge filtration is the restriction of the extended Hodge bundles\footnote{See (\ref{item:hodge}) in Section~\ref{sec:vmhs}.} to $V_p$. The rational structure depends on the tangent vector $\vv$. Details of its construction can be found in \cite{steenbrink-zucker} in the case when $X$ is a curve. In general, following Kashiwara \cite{kashiwara}, one reduces to the 1-dimensional case by restricting to a curve $C$ in $X$ that passes through $p$ and is tangent to $\vv$. The weight filtration $W_\bdot$ of $\bV$ restricts to a weight filtration $W_\bdot$ on $V_p$. In general, the weight filtration of the limit MHS $V_\vv$, denoted $M_\bdot$, is not equal to $W_\bdot$. Rather, it is the {\em relative weight filtration} associated to the action of the residue of the connection on $(V_p,W_\bdot)$. See \cite{steenbrink-zucker} for definitions and details. However, when the global monodromy action $\pi_1(X,x) \to \Aut V_x$ is unipotent, the relative weight filtration $M_\bdot$ equals $W_\bdot$. (Cf.\ \cite[p.~85]{hain-zucker}.) That is the case for all of the variations of MHS considered in this paper.

Finally, the limit MHS associated to $\vv \in T_p \Xbar$ themselves form an admissible VMHS over
$$
T_p \Xbar - \bigcup_{p\in D_\alpha} T_p D_\alpha,
$$
where $D_\alpha$ ranges over the local components of $D$ at $p$.

\section{Proof of Theorems~\ref{thm:free} and \ref{thm:naive}}
\label{sec:proofs}

\subsection{Proof of Theorem~\ref{thm:free}}

We prove the theorem by showing that each map in the diagram in the statement of Proposition~\ref{prop:dual_goldman} is a morphism of MHS. We use the notation of the proposition.

The first map is a morphism of MHS by Corollary~\ref{cor:restriction}. Theorem~\ref{thm:PD} implies that the second and fourth maps are isomorphisms of MHS. The third map is a morphism of MHS because it is the direct limit of the duals to the cup products
$$
H^1(\Xbar,Y;\bL/\bI^n)^{\otimes 2} \to H^2(\Xbar,Y;\bL/\bI^n)
$$
induced by the product pairing $\bL^{\otimes 2} \to \bL$, which is a morphism of variations of MHS. This leaves $\cs$.

\begin{lemma}
\label{lem:dual_cs_mhs}
The dual $\csdual : H^1(X;\Ldual_X) \to \Hdual^0(\Lambda X)$ of the Chas--Sullivan map is a morphisms of MHS.
\end{lemma}

\begin{proof}
It is clearly defined over $\Q$. So it suffices to show that it preserves the Hodge and weight filtrations. But this follows from the formula for $\cs$ in Proposition~\ref{prop:formula_bcs} and the definition \cite[\S3.2]{hain:dht} of the Hodge and weight filtrations on the bar construction on a mixed Hodge complex.
\end{proof}

\subsection{Proof of Theorem~\ref{thm:naive}}

We prove the theorem by showing that each map in the diagram in the statement of Proposition~\ref{prop:dual_kk} is a morphism of MHS. As in the proof of Theorem~\ref{thm:free}, the first four maps (going counter clockwise) in the diagram in the statement of the proposition are morphisms of MHS. To complete the proof, we need to show that $\csp$ is a morphism.

\begin{lemma}
The dual $\csp : H^1(\Xbar,Y\cup\{x_0,x_1\};\Pdual_{X,x_0,x_1}) \to \Hdual^0(P_{x_0,x_1} X)$ of the Chas--Sullivan map is a morphisms of MHS.
\end{lemma}

\begin{proof}
The proof is similar to the proof of Lemma~\ref{lem:dual_cs_mhs} above. The only point that may need addressing is to explain why
$$
\csp^\ast : H^1(X,Y\cup\{x_0,x_1\};\sP_{x_0,x_1}) \to H^0_\dR(P_{x_0,x_1}X)
$$
is a morphism of MHS. The formula (\ref{eqn:formula_bkk}) and the fact that a cone of mixed Hodge complexes is a mixed Hodge complex \cite{deligne:hodge2} imply that
$$
\csp^\ast : H^1(X,\{x_0,x_1\};\sP_{x_0,x_1}) \to H^0_\dR(P_{x_0,x_1}X)
$$
preserves the Hodge and weight filtrations. Since it is defined over $\Q$, it is also a morphism of MHS. Naturality implies that
$$
H^1(X,Y\cup\{x_0,x_1\};\sP_{x_0,x_1}) \To H^1(X,\{x_0,x_1\};\sP_{x_0,x_1})
$$
is a morphism of ind-MHS. This completes the proof.
\end{proof}

\begin{remark}
\label{rem:restn}
If $U$ is a Zariski open subset of $X-\{x_0,x_1\}$, then the map $H_0(\Lambda U)^\wedge \to H_0(\Lambda X')^\wedge$ induced by the inclusion is a morphism of MHS. It follows that the composition
$$
H_0(\Lambda U)^\wedge \otimes H_0(P_{x_0,x_1}X)^\wedge \to H_0(P_{x_0,x_1}X)^\wedge
$$
of this with $\csp$ is also a morphism of MHS.
\end{remark}

\section{Proof of Theorem~\ref{thm:action}}
\label{sec:proof_technical}


\subsection{Boundary Components and Tangent Vectors}
\label{sec:tangent_vecs}

Complex algebraic curves do not have boundary components. In order to apply Hodge theory to study path torsors of surfaces with boundary, we need to explain the algebraic equivalent, ``tangent vectors at infinity''.

\subsubsection{Real oriented blowups}

Suppose that $S$ is a real 2-manifold and that $\Sigma$ a finite subset of $S-\partial S$. The real oriented blow up of $S$ at $\Sigma$
$$
\pi : \BlR{\Sigma} S \to S
$$
replaces each $p\in \Sigma$ by the circle of oriented rays
$$
\pi^{-1}(p) = \big((T_p S) - \{0\}\big)/\R_+
$$
in the tangent space of $S$ at $p$. It is a manifold with one additional boundary circle for each $p\in \Sigma$. For example, the real oriented blow up
$$
\BlR{0} \D \to \D
$$
of the disk $\D$ at the origin is the annulus $S^1 \times [0,1]$; the projection takes $(r,\theta)\in S^1\times [0,1]$ to $r^{i\theta}\in\D$.

Denote by $[\vv]$ the point of $\Bl_p S$ that corresponds to the oriented ray in $T_p S$ spanned by the non-zero vector $\vv \in T_p $. Since the natural inclusion $S-\Sigma \hookrightarrow \BlR{\Sigma} S$ is a homotopy equivalence, every local system $\bV$ on $S-\Sigma$ extends canonically to a local system on $\BlR{\Sigma} S$. Denote its fiber over $[\vv]$ by $V_\vv$. It depends only on $[\vv]$. The local monodromy operator at $p$ is the action of parallel translation around the corresponding boundary circle of $\BlR{\Sigma} S$. It acts on $V_\vv$.

One can also form the augmented surface
$$
\widehat{S}_\Sigma := \BlR{\Sigma} \cup \bigcup_{p\in \Sigma} \BlR{0} T_p S,
$$
where $[\vv] \in \BlR{p}$ is identified with $[\vv]\in \BlR{0} T_p S$. It retracts onto $\BlR{\Sigma} S$.

\subsubsection{Tangential base points}
\label{sec:tang_bps}
Suppose that $S^c$ is a topological surface and that $p_0,p_1\in S^c$ are not necessarily distinct points. Suppose that $\vv_j \in T_{p_j} S$, $j=0,1$, are non-zero tangent vectors. Set $S = S^c-\{p_0,p_1\}$ and $\Shat := \Bl_{p_0,p_1}S^c$. The inclusion $S \hookrightarrow \Shat$ is a homotopy equivalence with dense image. Define $P_{\vv_0,\vv_1} S$ to be $P_{[\vv_0],[\vv_1]} \Shat$. For $x_0,x_1\in S$, we define (following Deligne \cite[\S15.9]{deligne:p1})
\begin{align*}
\pi(S;\vv_0,\vv_1) &:= \pi(\Bl_{p_0,p_1}S^c;[\vv_0],[\vv_1]) \cr
\pi(S;x_0,\vv_1) &:= \pi(\Bl_{p_0}S^c-\{p_1\};\vv_0,x_1), \cr
\pi(S;\vv_0,x_1) &:= \pi(\Bl_{p_1}S^c-\{p_0\};x_0,[\vv_1]).
\end{align*}
When $p_0 = p_1$ and $\vv = \vv_0 = \vv_1$, define $\pi_1(S,\vv)$ to be $\pi(S;\vv,\vv)$.
\begin{center}
\begin{figure}[!ht]
\epsfig{file=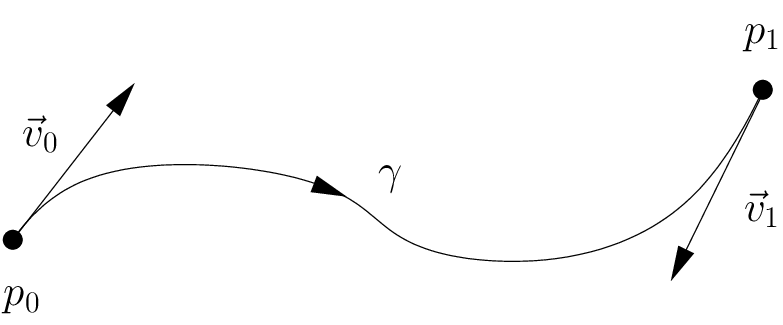, width=2in}
\caption{A path from $\vv_0$ to $\vv_1$}
\label{fig:path}
\end{figure}
\end{center}

\subsubsection{Extended local systems and their pairings}

We continue with the notation of the previous section (\ref{sec:tang_bps}). To simplify exposition, we suppose that $p_0$ and $p_1$ are distinct. Denote the diagonal in $T\times T$ by $\Delta_T$. Denote the local system over $S^2-\Delta_S$ with fiber $\Hdual^0(P_{x_0,x_1}S)$ over $(x_0,x_1)\in S^2$ by $\Pdual_{S\times S}$. It extends canonically to a local system over $\Shat^2 - \Delta_\Shat$. Its fiber over $a_j \in \{x_j,\vv_j\}$ is
$$
\Hdual^0(P_{a_0,a_1}S;\Q) = \Hom^\cts_\Q(\Q\pi(S;a_0,a_1),\Q).
$$

We also need to consider the local system $\bLambda_{S^2}$ over $S^2 - \Delta_S$ whose fiber over $(x_0,x_1)$, $x_0\neq x_1$, is
$$
\Hdual^0(\Lambda (S-\{x_0,x_1\})).
$$
This extends to a local system on $\Shat^2 - \Delta_\Shat$. Its fiber over $([\vv_0],[\vv_1])$ will be denoted by $\Hdual^0(\Lambda S'_{\vv_0,\vv_1})$. The surjection $\Lambda (S-\{x_0,x_1\}) \to \Lambda S$ induces an inclusion
$$
\Hdual^0(\Lambda S) \hookrightarrow \Hdual^0(\Lambda (S-\{x_0,x_1\}))
$$
for all $(x_0,x_1)$, and therefore an inclusion $\Hdual^0(\Lambda S)_{S^2-\Delta_S} \hookrightarrow \bLambda_{S^2}$ of the constant local system over $S^2-\Delta_S$ with fiber $\Hdual^0(\Lambda S)$ into $\bLambda_{S^2}$ and therefore an inclusion
\begin{equation}
\label{eqn:canon_inclusion}
\Hdual^0(\Lambda S) \hookrightarrow \Hdual^0(\Lambda S'_{\vv_0,\vv_1}).
\end{equation}

The dual of the KK-action
$$
\kappa^\ast : \Hdual^0(P_{x_0,x_1} S) \To \Hdual^0(\Lambda (S-\{x_0,x_1\})) \otimes \Hdual^0(P_{x_0,x_1} S)
$$
induces a coaction $j^\ast\Pdual_{S^2} \to \bLambda_{S^2} \otimes j^\ast\Pdual_{S^2}$ of local systems over $S^2-\Delta_S$ and their extensions to $\Shat^2 - \Delta_\Shat$. In particular, it induces a limit coaction
\begin{equation}
\label{eqn:coaction}
\Hdual^0(P_{\vv_0,\vv_1} S) \To \Hdual^0(\Lambda S'_{\vv_0,\vv_1}) \otimes \Hdual^0(P_{\vv_0,\vv_1} S)
\end{equation}
on the fiber over $([\vv_0],[\vv_1])$.

\subsection{Relation to limit MHS}
\label{sec:limitMHS}

Suppose that $\Xbar$ is a compact Riemann surface and that $D$ is a non-empty finite subset of $\Xbar$. Set $X=\Xbar-D$. Suppose that $p_0,p_1\in D$ and that $\vv_0 \in T_{p_0} \Xbar$ and $\vv_1 \in T_{p_1}\Xbar$ are non-zero tangent vectors.\footnote{The case where $p_0=p_1$ is similar. Instead of using the variation of MHS $\Pdual$ over $X^2-\Delta_X$, one uses the variation $\Ldual$ over $X$. Details are left to the reader.} We use the notation of the previous section.

Recall from Theorem~\ref{thm:HZ} that $\Pdual_{X\times X}$ is in ind-$\MHS(X\times X)$.

\begin{proposition}
\label{prop:weak_result}
The local system $\bLambda_{X\times X}$ underlies an object of ind-$\MHS(X^2 - \Delta_X)$ and the dual of the KK-action
$$
\kappa^\ast: \Hdual^0(P_{\vv_0,\vv_1} X) \To \Hdual^0(\Lambda X'_{\vv_0,\vv_1}) \otimes \Hdual^0(P_{\vv_0,\vv_1} X)
$$
is a morphism in ind-$\MHS(X^2-\Delta_X)$.
\end{proposition}

\begin{proof}
Denote the coordinates in $X^3$ by $(x_0,x_1,x_2)$. Denote the diagonal $x_j=x_k$ by $\Delta_{j,k}$. Set
$$
Z = X\times X \times X - (\Delta_{0,1} \cup \Delta_{0,2} \cup \Delta_{1,2}).
$$
The projection $\pi : Z \to X^2-\Delta_X$ along the 3rd factor $\pi : (x_0,x_1,x_2) \mapsto (x_0,x_1)$ is a topologically locally trivial family of affine curves. Fix a point $p\in X$. Then one has the prounipotent local system over $Y := (X-\{p\})^2-\Delta$ whose fiber over $(x_0,x_1)$ is
$$
\Q\pi_1(X-\{x_0,x_1\},p)^\wedge.
$$
It underlies a pro-object of $\MHS(Y)$. This is an immediate consequence of \cite[Thm.~1.5.1]{hain:dht2}. Alternatively, it follows from the main theorem of \cite{hain-zucker} as follows: Let $Z_Y$ be the inverse image of $Y$. Choose a point $q\in X^2$ such that $(q,p)\in Z_Y$. Then one has the sequence
\begin{equation}
\label{eqn:gp_extn}
\pi_1^\un(Z_{Y,p},(q,p)) \to \pi_1^\un(Z_Y,(q,p)) \to \pi_1^\un(Y,p) \to 1
\end{equation}
of unipotent completions. It is well-know and easy to see that it is also exact on the left. This is equivalent to the fact that
$$
\Q\pi_1(Z_{Y,p},(q,p))^\wedge \to \Q\pi_1(Z_Y,(q,p))^\wedge
$$
is injective. The section $p$ induces a splitting of (\ref{eqn:gp_extn}). It is a morphism of pro-MHS. It follows that the monodromy action
$$
\Q \pi_1^\un(Y,p)^\wedge \to \End \Q\pi_1(Z_{Y,p},(q,p))^\wedge
$$
is a morphism of pro-MHS, which implies that the local system is a pro-object of $\MHS(Y)$.

Since $\bLambda_{X^2}$ is a quotient of this local system, and since the quotient map is a morphism of MHS on each fiber, it follows that $\bLambda_{X^2}$ is in ind-$\MHS(Y)$. But since it extends as a local system to $X^2-\Delta_X$, it is also in ind-$\MHS(X^2-\Delta_X)$.

The second assertion follows as the fiberwise dual KK-action is a morphism
$$
\Pdual_{X\times X} \to \bLambda_{X^2}\otimes \Pdual_{X\times X}
$$
of local systems over $X^2-\Delta_X$ and, on each fiber, a morphism of MHS by Theorem~\ref{thm:naive}. It is therefore a morphism in ind-$\MHS(X^2-\Delta_X)$ and, consequently, induces a morphism of limit MHS on the fiber over $(\vv_0,\vv_1)$.
\end{proof}

\begin{corollary}
\label{cor:inclusion}
There is a canonical MHS on $\Hdual^0(\Lambda X'_{\vv_0,\vv_1})$ and the canonical inclusion (\ref{eqn:canon_inclusion}) $\Hdual^0(\Lambda X) \hookrightarrow \Hdual^0(\Lambda X'_{\vv_0,\vv_1})$ is a morphism of ind-MHS.
\end{corollary}

\begin{proof}
The MHS on $\Hdual^0(\Lambda X'_{\vv_0,\vv_1})$ is the limit MHS constructed in the proof above. By the functoriality of the MHS on the continuous cohomology of free loop spaces, for each $(x_0,x_1) \in X^2-\Delta_X$, the inclusion $X-\{x_0,x_1\} \to X$ of varieties induces a morphism of MHS
$$
\Hdual^0(\Lambda X) \to \Hdual^0(\Lambda(X-\{x_0,x_1\})).
$$
This induces a morphism from the constant variation over $X^2-\Delta_X$ to the variation whose fiber is $\Hdual^0(\Lambda(X-\{x_0,x_1\}))$. It therefore induces a morphism of MHS from $\Hdual^0(\Lambda X)$ into the limit MHS $\Hdual^0(\Lambda X'_{\vv_0,\vv_1})$.
\end{proof}

\subsection{The splitting theorem}
\label{sec:splitting}

We continue using the notation of the preceding sections. Choose disjoint imbedded closed disks $B_0$ and $B_1$ in $\Xbar$ such that $B_j\cap D = \{p_j\}$. Set $U = X\setminus (B_0\cup B_1)$. When $p_0\in B_0$ and $p_1 \in B_1$, the inclusion $U\hookrightarrow X$ is a homotopy equivalence and thus induces a homotopy equivalence $\Lambda U \to \Lambda X$ which factors through $\Lambda(X-\{p_0,p_1\})$. Consequently, $\Lambda X$ is a retract of $\Lambda(X-\{p_0,p_1\})$ and the composite
$$
\Hdual^0(\Lambda X) \to \Hdual^0(\Lambda (X-\{p_0,p_1\})) \to \Hdual^0(\Lambda U) \cong \Hdual^0(\Lambda X)
$$
is the identity. Taking the limit as $p_j \to [\vv_j]$, one obtains the sequence
$$
\Hdual^0(\Lambda X) \to \Hdual^0(\Lambda X'_{\vv_0,\vv_1}) \to \Hdual^0(\Lambda X).
$$
whose composite is the identity. The left-hand map is a morphism of MHS by Corollary~\ref{cor:inclusion}. It is not clear that the right-hand map is a morphism of MHS as it is not in the fiber over a general point $(x_0,x_1) \in X^2-\Delta_X$.\footnote{This follows from the fact that the inclusion in Corollary~\ref{cor:inclusion} is not split in the category of local systems, and therefore cannot be split in ind-$\MHS(X^2-\Delta_X)$ by the Theorem of the Fixed Part.}

The following ``splitting theorem'' is proved in Section~\ref{sec:proof}.

\begin{theorem}
\label{thm:splitting}
The projection
\begin{equation}
\label{eqn:projn}
\Hdual^0(\Lambda X'_{\vv_0,\vv_1}) \to \Hdual^0(\Lambda X)
\end{equation}
is a morphism of MHS.
\end{theorem}

\subsubsection{Theorem~\ref{thm:splitting} implies Theorem~\ref{thm:action}}

The KK-action
$$
\kappa : H_0(\Lambda X)^\wedge \otimes H_0(P_{\vv_0,\vv_1} X) \to H_0(P_{\vv_0,\vv_1} X)
$$
is the restriction of the action
$$
H_0(\Lambda X'_{\vv_0,\vv_1})^\wedge \otimes H_0(P_{\vv_0,\vv_1} X)^\wedge \to H_0(P_{\vv_0,\vv_1} X)^\wedge
$$
to $H_0(\Lambda U)^\wedge \cong H_0(\Lambda X)^\wedge$. The dual action is thus obtained from the coaction (\ref{eqn:coaction}) by pushing out along the projection (\ref{eqn:projn}). The result follows as both are morphisms of MHS by Theorems~\ref{thm:naive} and \ref{thm:splitting}.

\subsection{Proof of Theorem~\ref{thm:splitting}}
\label{sec:proof}

In this section, we assume that $p_0$ and $p_1$ are distinct. We use the notation and setup of Sections~\ref{sec:limitMHS} and \ref{sec:splitting}.  To prove the theorem, we need to prove that when we give $H_0(\Lambda X')^\wedge$ the limit MHS $H_0(\Lambda X'_{\vv_0,\vv_1})^\wedge$, the splitting
\begin{equation}
\label{eqn:splitting}
\Hdual^0(\Lambda X'_{\vv_0,\vv_1}) \to \Hdual^0(\Lambda X)
\end{equation}
induced by the inclusion $U \hookrightarrow X-\{x_0,x_1\}$ (when $x_j\in B_j-\{p_j\}$) is a morphism of ind-MHS. This will imply that $H_0(\Lambda X)^\wedge$ is a direct summand of $H_0(\Lambda X'_{\vv_0,\vv_1})^\wedge$ in ind-$\MHS$. We do this using the construction of the limit MHS in \cite{hain:dht2}. This requires familiarity with the constructions in \cite{hain:dht,hain:dht2}.

The first observation is that, since the variation over
$$
(T_{p_0}\Xbar-\{0\}) \times (T_{p_1}\Xbar-\{0\})
$$
whose fiber over $(\vv_0,\vv_1)$ is $\Hdual^0(\Lambda X'_{\vv_0,\vv_1})^\wedge$ is a nilpotent orbit of ind-variations of MHS, it suffices to prove the result for $\vv_j$ in a punctured neighbourhood of $0 \in T_{p_j}\Xbar$.

To compute the limit MHS on invariants of $X-\{x_0,x_1\}$ as $x_j \to p_j$ along $\vv_j$ for $j=0,1$, choose local holomorphic parameters $t_j:U_j \overset{\simeq}{\To} \D$ on $\Xbar$ centered at $p_j$ such that $\vv_j = \partial/\partial t_j$. This is possible because we can, as remarked above, take $\vv_0$ and $\vv_1$ to be as small as we like.

Denote $t_j^{-1} : U_j \to \D$ by $x_j$. Consider the pullback
$$
\xymatrix{
\Xbar \times \D \ar[r]\ar[d] & \Xbar\times U_0\times U_1 \ar[r]\ar[d] & \Xbar \times \Xbar\times \Xbar \ar[d] \cr
\D \ar[r]^(.45){(x_0,x_1)} & U_0\times U_1 \ar[r] & \Xbar\times \Xbar
}
$$
of the trivial family
$$
\Xbar \times \Xbar \times \Xbar \to \Xbar \times \Xbar,\qquad (x,x_0,x_1)\mapsto (x_0,x_1)
$$
to $\D$. For $S\subset \Xbar$, set $D_S = D\times S$. Set
$$
Y = \Xbar\times \Xbar - (D_\Xbar \cup \Delta_\Xbar).
$$
The restriction $Y'$ of the projection $Y \to \Xbar\times\Xbar$ to the punctured disk $\D'$ is a topologically locally trivial fiber bundle whose fiber over $t$ is $X-\{x_0(t),x_1(t)\}$.

The next step is to write the restriction $Y_\D$ to $\D$ as the complement of a normal crossing divisor. We do this by blowing up the points $(p_0,0)$ and $(p_1,0)$ of $\Xbar\times \D$. Denote the blowup by $\phi : Z \to \Xbar\times\D$ and the projection of $Z$ to the disk by $\pi : Z \to \D$. The fiber over $t\in \D$ will be denoted by $Z_t$. Note that $Z_t=\Xbar$ when $t\neq 0$. The special fiber $Z_0$ has 3 components: the strict transform of $\Xbar\times \{0\}$, which, by abuse of notation, we will denote by $\Xbar$, and the two exceptional divisors $E_0$, $E_1$ that lie over $p_0$, $p_1$, respectively.
\begin{center}
\begin{figure}[!ht]
\epsfig{file=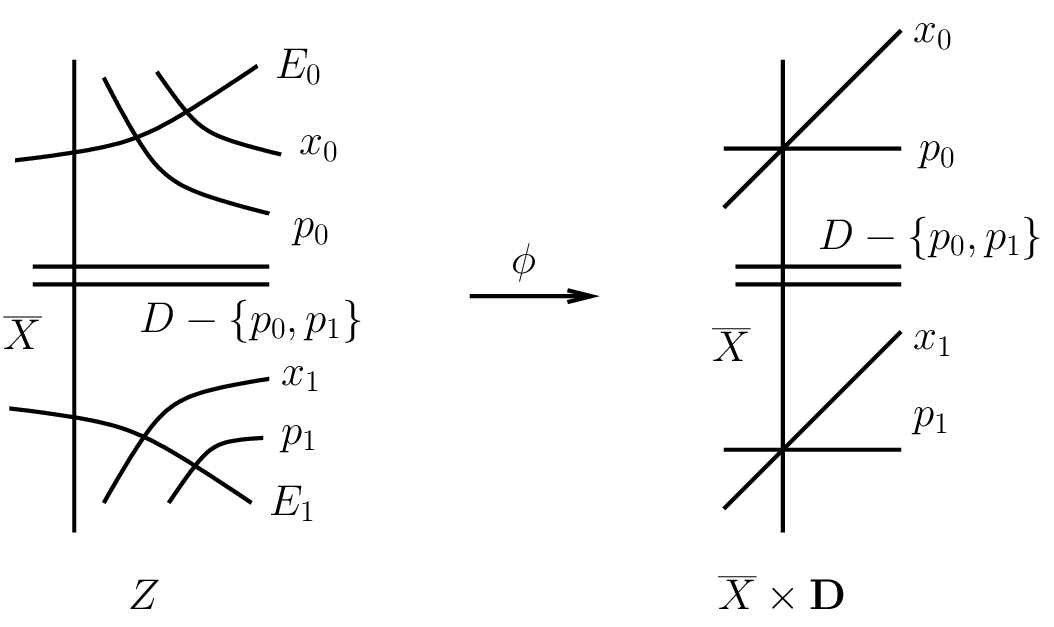, width=2.75in}
\label{fig:blowup}
\caption{The map $\phi : Z \to \Xbar\times\D$}
\end{figure}
\end{center}

\begin{remark}
One can think of $X_{\vv_0,\vv_1}'$ as the union of $\BlR{p_0,p_1}\Xbar - (D-\{p_0,p_1\})$ with the real oriented blow ups $\BlR{E_j\cap \Xbar}E_j -\{p_0,x_j\}$ of the exceptional divisors at the double points. The identification of the boundary circle of $\BlR{E_j\cap \Xbar} (E_j -\{p_0,x_j\})$ with the boundary circle of $\BlR{p_0,p_1}\Xbar$ above $p_j\in \Xbar$ is determined by $\vv_j$ and the Hessian of $\pi : Z \to \D$ at $p_j\in Z_0$.
\end{remark}

As in \cite{hain:dht2}, we choose neighbourhoods $N_\Xbar$ of $\Xbar$, $N_j$ of $E_j$ such that the inclusions
$$
\Xbar \hookrightarrow N_\Xbar,\quad E_j \hookrightarrow N_j,\quad \{p_j\} \hookrightarrow N_j\cap N_\Xbar,\quad j=0,1
$$
are homotopy equivalences. Then $N' \hookrightarrow Y'$ is a homotopy equivalence. We can further assume that there is a $\delta> 0$ such that when $0<|t|<\delta$, the intersection $Z_t\cap N_\Xbar$ is $\Xbar$ with two small disks removed, one about $p_0$ and the other about $p_1$. This implies that the inclusion $Z_t\cap N_\Xbar \hookrightarrow \Xbar-\{p_0,p_1\}$ is a homotopy equivalence. By rescaling $t$ (and thus $\vv_0$ and $\vv_1$ as well), we may assume that $\delta > 1$, so that $Z_1$ is contained in $N_\Xbar$ and the inclusion $(Z_1\cap N_\Xbar)-D\hookrightarrow X$ is a homotopy equivalence.
\begin{center}
\begin{figure}[!ht]
\epsfig{file=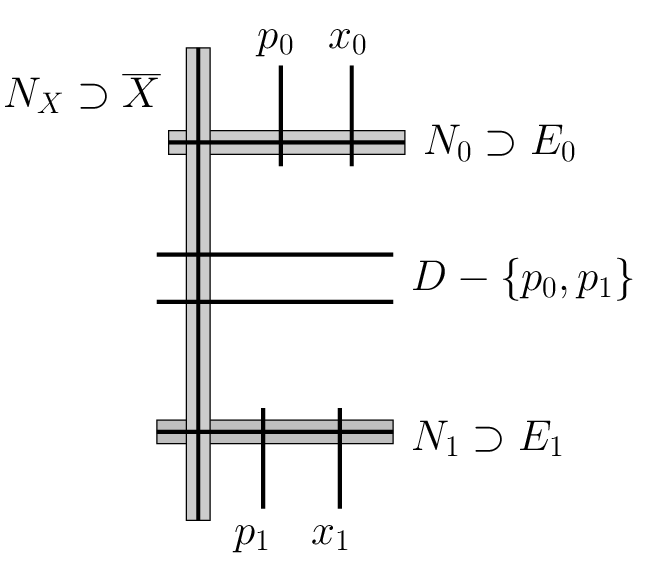, width=2in}
\label{fig:covering}
\caption{The simplicial neighbourhood $N_\bdot$ of $Z_0$}
\end{figure}
\end{center}

We will use the generic notation $\sD^\bdot(T)$ to denote a de~Rham mixed Hodge complex (see \cite[\S4.1]{hain:dht} for the definition) that computes the cohomology of a topological space $T$. For example, the standard de~Rham mixed Hodge complex$\sD^\bdot(X)$ that computes the cohomology of $X = \Xbar -D$ has complex part the $C^\infty$ log complex $E^\bdot(\Xbar\log D)$. (Cf.\ \cite[\S5.6]{hain:dht} and \cite{navarro}.)

Natural de~Rham mixed Hodge complexes (MHC) for $N'$, $N_\Xbar'$, $\D'$ are constructed in \cite[\S1.2]{hain:dht2}. The projection $N_\Xbar \to X$ and the maps $N_\Xbar' \to N'\to \D'$ induced by $\pi$ induce maps of de~Rham MHC, as does the inclusion $\{1\}\hookrightarrow \D'$. These morphisms fit into the commutative diagram
$$
\xymatrix@R=4pt{
\sD^\bdot(X) \ar[r] &\sD^\bdot(N_\Xbar') \ar[dd] \cr
&& \sD^\bdot(\D') \ar[ul]\ar[dl] \ar[r]^\e & \Q \cr
  & \sD^\bdot(N')
}
$$
of de~Rham MHC. One can therefore form the 2-sided bar constructions
$$
\sD^\bdot(X_{\vv_0,\vv_1}') := B(\sD^\bdot(N'),\sD^\bdot(\D'),\Q) \text{ and } B(\sD^\bdot(N_\Xbar'),\sD^\bdot(\D'),\Q).
$$
These are MHC (by \cite[\S3.2]{hain:dht}) which compute the cohomology of the homotopy fibers of $N_\Xbar' \to \D'$ and $N_\Xbar' \to \D'$, respectively. That is, they compute
$$
H^\bdot(X-\{x_0(1),x_1(1)\}) \text{ and } H^\bdot(X),
$$
respectively. As shown in \cite{hain:dht2}, the first complex computes the limit MHS $H^\bdot(X_{\vv_0,\vv_1}')$. The morphism $\sD^\bdot(X) \to \sD^\bdot(N_\Xbar')$ induces a quasi-isomorphism
\begin{equation}
\label{eqn:qism}
\sD^\bdot(X) \to B(\sD^\bdot(N_\Xbar'),\sD^\bdot(\D'),\Q)
\end{equation}
of MHC. So the second computes the MHS on $H^\bdot(X)$.

The limit MHS on $\Hdual^0(\Lambda X'_{\vv_0,\vv_1})$ is computed by the circular bar construction
$$
B(\sD^\bdot(X_{\vv_0,\vv_1}');\sD^\bdot(X_{\vv_0,\vv_1}')),
$$
which computes the correct cohomology by Theorem~\ref{thm:pullback_DR}. The quasi-isomorphism (\ref{eqn:qism}) induces a quasi-isomorphism on circular bar constructions which is a morphism of MHC. It follows that
$$
B\big(B(\sD^\bdot(N_\Xbar'),\sD^\bdot(\D'),\Q);B(\sD^\bdot(N_\Xbar'),\sD^\bdot(\D'),\Q)\big)
$$
computes the MHS on $\Hdual^0(\Lambda X)$. The morphism $\sD^\bdot(N') \to \sD^\bdot(N_\Xbar)$ induces a morphism
$$
B(\sD^\bdot(X_{\vv_0,\vv_1}');\sD^\bdot(X_{\vv_0,\vv_1}')) \to
B\big(B(\sD^\bdot(N_\Xbar'),\sD^\bdot(\D'),\Q);B(\sD^\bdot(N_\Xbar'),\sD^\bdot(\D'),\Q)\big)
$$
which induces the splitting (\ref{eqn:splitting}). This implies that the splitting is, in fact, a morphism of MHS.

\section{Mapping class group actions}
\label{sec:mcg}

In this section we prove Theorem~\ref{thm:mcg} and, as we finally have all the pieces assembled, Theorem~\ref{thm:splittings}.

\subsection{Mapping class groups and their relative completions}

For a compact oriented surface $S$ and a collection $V=\{\vv_p \in T_p S : p\in A\}$ of non-zero tangent vectors indexed by a finite subset $A$ of $S$, we have the mapping class group
$$
\Gamma_{S,V} := \pi_0 \Diff^+(S,V)
$$
of isotopy classes of diffeomorphisms of $S$ that fix $V$ pointwise. This depends only on the genus $g$ of $S$ and the number $n$ of tangent vectors. When the particular surface is not important, we will denote it by $\G_{g,\vec{n}}$.\footnote{We also have the mapping class group $\Gamma_{S,\partial S} := \pi_0 \Diff^+ (S,\partial S)$ of isotopy classes of orientation preserving diffeomorphisms of $S$ that fix the boundary. There is a canonical isomorphism $\Gamma_{S,V} \cong \Gamma_{\Shat,\partial \Shat}$, where $\Shat := \BlR{A} S$ denotes the real oriented blow up of $S$ at $A$. Because of our interest in applications of Hodge theory to algebraic curves, we will use the tangent vector formulation rather than the more traditional boundary component formulation.}

For a commutative ring $R$, set $H_R = H_1(S;R)$. Let $\Sp(H)$ be the $\Q$-affine algebraic group whose group of $F$ rational points ($F$ a $\Q$-algebra) is the subgroup of $\Aut H_F$ that fixes the intersection form. The action of the mapping class group on $S$ induces a homomorphism
$$
\rho : \G_{S,V} \to \Sp(H_\Q).
$$
Denote the completion of $\G_{S,V}$ with respect to $\rho$ by $\cG_{S,V}$. (See \cite{hain:torelli} for the definition and basic properties.) This is an affine $\Q$ group that is an extension
$$
1 \to \U_{S,V} \to \cG_{S,V} \to \Sp(H) \to 1
$$
of $\Sp(H)$ by a prounipotent group $\U_{S,V}$. There is a canonical, Zariski dense homomorphism
$$
\rhotilde : \Gamma_{S,V} \to \cG_{S,V}(\Q).
$$
Denote the Lie algebras of $\cG_{S,V}$, $\U_{S,V}$ and $\Sp(H)$ by $\g_{S,V}$, $\u_{S,V}$ and $\sp(H)$, respectively.

Suppose that $X$ is a compact Riemann surface of genus $g$ and $V$ is a set of $n$ non-zero tangent vectors on $X$ anchored at $n$ distinct points. By the main result of \cite{hain:torelli}, $\cG_{X,V}$ and $\U_{X,V}$ have natural mixed Hodge structures. More precisely, their coordinate rings $\cO(\cG_{X,V})$ and $\cO(\U_{X,V})$ are Hopf algebras in ind-$\MHS$. This implies that their Lie algebras are Lie algebras in pro-$\MHS$. Furthermore,
$$
\g_{X,V} = W_0 \g_{X,V},\ \u_{X,V} = W_{-1}\g_{X,V} \text{ and } \sp(H) = \Gr^W_0 \g_{X,V}.
$$

When $2g-2+n>0$ one has the moduli space $\M_{g,\vec{n}}$ of compact Riemann surfaces with $n$ non-zero tangent vectors. It will be regarded as a complex analytic orbifold. The local system over $\M_{g,\vec{n}}$ with fiber $\g_{X,V}$ over $(X,V)$ is a pro-object of $\MHS(\M_{g,\vec{n}})$.

\subsection{The Lie algebra homomorphism $\g_{S,V} \to H_0(\Lambda S')^\wedge$}
\label{sec:lie}

As above, we suppose that $S$ is a compact oriented surface, that $A$ is a finite subset of $S$. Set $S'=S-A$. Choose $p_o\in A$ and a non-zero tangent vector $\vv_o \in T_{p_o} S$. A continuous endomorphism $\delta$ of $\Q\pi_1(S',\vv_o)^\wedge$ is a derivation if
$$
\delta(\gamma\mu) = \delta(\gamma)\mu + \gamma\delta(\mu).
$$

Denote by $\sigma_o$ the element of $\pi_1(S',\vv_o)$ that rotates the base point vector once about $p$ in the positive direction. (In other words, this is the image of the positive generator of the fundamental group of the boundary component of $\BlR{p_o} S$.) Denote the set of {\em continuous} derivations of $\Q\pi_1(S',\vv_o)^\wedge$ that vanish on $\sigma_o$ by
$$
\Der^\theta \Q\pi_1(S',\vv_o)^\wedge.
$$
The KK-action induces induces the continuous Lie algebra homomorphism
$$
\kappahat_o : H_0(\Lambda S')^\wedge \To \Der^\theta \Q\pi_1(S',\vv_o)^\wedge
$$
defined by $\kappahat(\alpha) : \gamma \mapsto \kappa(\alpha \otimes \gamma)$. Kawazumi and Kuno have shown \cite[Thm.~4.13]{kk:johnson} that, when $A$ consists of a single point, it is surjective and has kernel spanned by the unit $1$. When $\#A>1$, they show \cite[Thm.~6.2.1]{kk:groupoid} that the direct sum
\begin{equation}
\label{eqn:A-action}
\kappahat_A : H_0(\Lambda S')^\wedge \to \bigoplus_{a\in A} \Der^\theta \Q\pi_1(S',\vv_a)^\wedge
\end{equation}
 of the $\kappahat_a$ has kernel spanned by the unit $1$, where each $\vv_a \in T_a S$ is non-zero.

\begin{remark}
Define a filtration $I^\bdot_\cD$ of $\Der^\theta \Q\pi_1(S',\vv_o)^\wedge$ by
$$
I^m_\cD \Der^\theta \Q\pi_1(S',\vv_o)^\wedge = \{\delta : \delta(I) \subset I^{m+1}\}
$$
where $I$ denotes the augmentation idea of $\Q\pi_1(S',\vv_o)^\wedge$. This satisfies
$$
\Der^\theta \Q\pi_1(S',\vv_o)^\wedge = I^{-1}_\cD \supseteq I^0_\cD \supseteq I^1_\cD \supseteq \cdots
$$
As we will see below, $I^{-1}_\cD/I^0_\cD$ is non-zero. 
\end{remark}

\begin{remark}
The Lie algebra $\p(S',\vv_o)$ of the unipotent (aka, Malcev) completion of $\pi_1(S',\vv_o)$ is the set of primitive elements of $\Q\pi_1(S',\vv_o)^\wedge$. The logarithm $\log \sigma_o$ of the boundary loop lies in $\p(S',\vv_o)$. The Lie algebra $\Der^\theta\p(S',\vv_o)$ of derivations that vanish on $\log\sigma_o$ is a Lie subalgebra of $\Der^\theta\Q\pi_1(S',\vv_o)^\wedge$.
\end{remark}

The action of $\G_{S,V}$ on $\pi_1(S',\vv_a)$ induces a homomorphism $\varphi_a : \cG_{S,V} \to \Aut \Q\pi(S\,\vv_a)^\wedge$ and a Lie algebra homomorphisms
\begin{equation}
\label{eqn:canon}
d\varphi_a : \g_{S,V} \to \Der^\theta\Q\pi_1(S',\vv_a)^\wedge.
\end{equation}

\begin{proposition}[Kawazumi--Kuno]
\label{prop:factorization}
There is a homomorphism $\phitilde : \g_{S,V} \to H_0(\Lambda S')^\wedge$ such that the diagram
$$
\xymatrix{
\g_{S,V} \ar[r]^(.45){d\varphi_a}\ar[d]_\phitilde & \Der^\theta \p(S',\vv_a)\ar[d]
\cr
H_0(\Lambda S')^\wedge \ar[r]^(.42){\kappahat_a} & \Der^\theta\Q\pi_1(S',\vv_a)^\wedge
}
$$
commutes for all $a\in A$.
\end{proposition}

\begin{proof}
The restriction
$$
\kappahat_A|_I : IH_0(\Lambda S')^\wedge \to
\bigoplus_{a\in A} I_\cD^1\Der^\theta\Q\pi_1(S',\vv_a)^\wedge
$$
of (\ref{eqn:A-action}) to $IH_0(\Lambda S')^\wedge$ is injective. To prove the result, it suffices to show that for each $a\in A$, the image of the inclusion
$$
d\varphi_a : \g_{S,V} \to I_\cD^1 \Der^\theta\Q\pi_1(S'\vv_a)^\wedge
$$
is contained in the image of $\kappahat_a|_I$. Once we have done this, we can define $\phitilde = (\kappahat_A|_I)^{-1}d\varphi_a$. It remains to prove prove that $\im d\varphi_a \subseteq \im \kappahat_a|_I$.

Denote the Dehn twist on the simple closed curve $C$ in $S'$ by $t_C$. Its image under the homomorphism $\rhotilde : \G_{S,V} \to \cG_{S,V}(\Q)$ lies in a prounipotent subgroup $U$, say. It therefore has a logarithm, which lies in the Lie algebra $\u$ of $U$. Since $\u\subseteq \g_{S,V}$, we have $\log t_C \in \g_{S,V}$. Since Dehn twist generate $\G_{S,V}$, and since the image of $\rhotilde : \G_{S,V} \to \cG_{S,V}(\Q)$ is Zariski dense, the Lie algebra $\g_{S,V}$ is generated as a topological Lie algebra by the set $\{\log\rhotilde(t_C) : C \text{ is a SCC in }S'\}$.

Kawazumi and Kuno \cite[Thm.~5.2.1]{kk:groupoid} have shown that for each $a\in A$,
$$
d\varphi_a(\log t_C) = \big(\log \kappahat_a(\alpha)\big)^2/2 :=
\frac{1}{2}\bigg(\sum_{n=1}^\infty \frac{1}{n}\,\kappahat_a\big(\psi_n(1-\alpha)\big)\bigg)^2,
$$
where $\alpha$ is a loop that represents $C$, which lies in the image of $\kappa_a$. This completes the proof.
\end{proof}

\begin{remark}
The Kawazumi--Kuno formula above implies that
$$
\phitilde(\log t_C) = (\log \alpha)^2/2 :=
\frac{1}{2}\bigg(\sum_{n=1}^\infty \frac{1}{n}\,\big(\psi_n(1-\alpha)\big)\bigg)^2,
$$
where $\alpha$ is a loop that represents $C$.
\end{remark}

\subsection{Proof of Theorem~\ref{thm:mcg}}

As in the Introduction, we suppose that $A$ is a finite subset of a compact Riemann surface $\Xbar$. Set $X=\Xbar-A$. Suppose that $V = \{\vv_a:a\in A\}$ is a collection of non-zero tangent vectors, where $\vv_a \in T_a \Xbar$. That 
$$
\phitilde : \g_{\Xbar,V} \to H_0(\Lambda X)^\wedge\otimes \Q(-1)
$$
is a morphism of MHS follows from the facts that the morphism $\kappa_A$ (see (\ref{eqn:A-action})) is an injective morphism of MHS and that each $d\varphi_a$ (see (\ref{eqn:canon}) )is a morphism of MHS.

\begin{remark}
As $(X,V)$ varies in $\M_{g,\vec{n}}$, $\phitilde$ defines a morphism in pro-$\MHS(\M_{g,\vec{n}})$.
\end{remark}


\subsection{Proof of Theorem~\ref{thm:splittings}} The existence of the splittings follows directly from Proposition~\ref{prop:splittings} and Remark~\ref{rem:splittings}. We will show that the splitting given by the lift $\chitilde$ gives a solution of $\KV^{(g,|A|)}$ as defined in \cite[Def.~4]{akkn1}. When $|A|=1$, a solution is a symplectic Magnus expansion of $\pi_1(X,\vv_a)$.

We first recall the condition $\KV^{(g,n+1)}$. Assume that  $A=\{0,\dots,n\}$. For each $k\in A$, let $\sigma_k$ be the loop in $\pi_1(X,\vv_k)$ that rotates the tangent vector $\vv_k$ once in the {\em negative} direction about $0 \in T_k\Xbar$. For $1\le k \le n$, choose a path $\mu_k$ in $X$ from $\vv_0$ to $\vv_k$. Let $\gamma_0 = \sigma_0$ and
$$
\gamma_k = \mu_k \sigma_k \mu_k^{-1} \qquad 1 \le j \le n.
$$
We may choose loops $\alpha_1,\dots,\alpha_g,\beta_1,\dots,\beta_g$ based at $\vv_0$ in $X$ and the paths $\mu_j$ such that $\pi_1(X,\vv_0)$ is freely generated by
$$
\{\alpha_j,\beta_j,\gamma_k : 1\le j \le g,\ 1 \le k \le n\}
$$
and
\begin{equation}
\label{eqn:reln0}
\gamma_0^{-1} = \prod_{j=1}^g (\alpha_j,\beta_j)\prod_{k=1}^n\gamma_k,
\end{equation}
where $(u,v)$ denotes the group commutator $uvu^{-1}v^{-1}$.

Taking $\alpha_j$ to $e^{x_j}$, $\beta_j$ to $e^{y_j}$ and $\gamma_k$ to $e^{z_k}$ when $k>0$ defines a complete Hopf algebra isomorphism
\begin{equation}
\label{eqn:theta_iso}
\Theta : \Q\pi_1(X,\vv_0)^\wedge \overset{\simeq}{\To}
\Q\ll x_1,\dots, x_g, y_1,\dots, y_g, z_1,\dots,z_n \rr
\end{equation}
onto the completed free associative algebra generated by the $x_j,y_j,z_k$. The Hopf algebra structure on the RHS is defined by declaring that each of the generators is primitive. The canonical weight filtration on the left hand side corresponds to the filtration of the RHS obtained by giving each $x_j$ and $y_j$ weight $-1$, and each $z_k$ weight $-2$.

The problem $\KV^{(g,n+1)}$ is to show that there is a continuous Hopf algebra automorphism $\Phi$ of the RHS of (\ref{eqn:theta_iso}) that preserves the natural weight filtration, acts trivially on its $W_\bdot$ associated graded, and satisfies
$$
\Phi\Big(\sum_{j=1}^g[x_j,y_j] + \sum_{k=1}^n z_k\Big) = \log\Big(\prod_{j=1}^g (e^{x_j}e^{y_j})\prod_{k=1}^n e^{z_k}\Big)
$$
and $\Phi(z_k) = g_k z_k g_k^{-1}$ for each $k\ge 1$, where the $g_k$ are group-like.

Observe that $\theta$ induces a Hopf algebra isomorphism
$$
\Thetabar : \Gr^W_\bdot \Q\pi_1(X,\vv_0)^\wedge \overset{\simeq}{\To}
\Q\langle x_1,\dots, x_g, y_1,\dots, y_g, z_1,\dots,z_n \rangle.
$$
The natural splitting of the weight filtration on $\Q\pi_1(X,\vv_0)^\wedge$ given by $\chitilde$ induces the complete Hopf algebra isomorphism
$$
\Theta^{\chitilde} : \Q\pi_1(X,\vv_0)^\wedge \to \Q\ll x_1,\dots, x_g, y_1,\dots, y_g, z_1,\dots,z_n \rr
$$
that is defined to be the composite
$$
\xymatrix@C=16pt{
\Q\pi_1(X,\vv_0)^\wedge \ar[r]^(.37)\simeq_(.37){\sigma^\chitilde} \ar@/^1.75pc/[rr]^{\Theta^{\chitilde}} & \prod_{n\ge 0}\Gr^W_\bdot \Q\pi_1(X,\vv_0) \ar[r]^(.4)\simeq_(.4){\Thetabar} & \Q\ll x_1,\dots, x_g, y_1,\dots, y_g, z_1,\dots,z_n \rr,
}
$$
where $\sigma^\chitilde$ is the natural isomorphism determined by $\chitilde$ that is defined in Proposition~\ref{prop:splittings}. We will show that $\Phi := \Thetabar \circ \sigma^\chitilde \circ (\Theta^{\chitilde})^{-1}$ is a solution of $\KV^{(g,n+1)}$. It suffices to show that
\begin{equation}
\label{eqn:gamma_k}
\Theta^\chitilde(\gamma_k) = g_k e^{z_k} g_k^{-1}
\end{equation}
when $k\ge 1$, where each $g_k$ is group-like, and that
\begin{equation}
\label{eqn:gamma_0}
\log\Theta^\chitilde(\gamma_0) = z_0 := -\Big(\sum_{j=1}^g[x_j,y_j] + \sum_{k=1}^n z_k\Big).
\end{equation}

To prove this, first observe that for all base points $a,b$ of $X$, the local system $\Gr^W_\bdot H_0(P_{a,b}X)^\wedge$ over $X\times X$ is constant. This implies that $\chitilde$ determines simultaneous isomorphisms
$$
\Theta_{k,m}^{\chitilde}: H_0(P_{\vv_k,\vv_m}X)^\wedge \overset{\simeq}{\To} \Q\ll x_1,\dots, x_g, y_1,\dots, y_g, z_1,\dots,z_n \rr
$$
for $1\le k,m\le n$. Since path multiplication induces a morphism of MHS, these are compatible with path multiplication. For $k\ge 1$, set
$
g_k = \Theta_{0,k}^\chitilde(\mu_k).
$
Then for $k\ge 1$
\begin{equation}
\label{eqn:conj}
\Theta^\chitilde(\gamma_k) = g_k\Theta^\chitilde_{k,k}(\sigma_k)^{-1} g_k^{-1}.
\end{equation}
 Note that equation~\ref{eqn:reln0} implies that $\log\Theta^\chitilde(\gamma_0) \equiv {z_0} \bmod W_{-3}$. To complete the proof, we show that $\Theta_{k,k}^\chitilde(\log\sigma_k) = z_k$ for all $k\ge 0$.

For each $k \ge 0$, the homomorphism
$
\Q\pi_1(\D',\partial/\partial t)^\wedge \to \Q\pi_1(X',\vv_k)^\wedge
$
induced by the inclusion $(\D',\partial/\partial t) \hookrightarrow (X',\vv_k)$ of an ``infinitesimal punctured neighbourhood'' of $k \in A$ is a morphism of MHS. It takes the negative generator $\sigma$ of $\pi_1(\D',\partial/\partial t)$ to $\sigma_k$. Since $\log \sigma$ spans a copy of $\Q(1)$ in $\Q\pi_1(\D^\ast,\partial/\partial t)^\wedge$, its image $\log \sigma_k$ spans a copy of $\Q(1)$ in $\Q\pi_1(X,\vv_k)^\wedge$. Since $\Theta_{k,k}^\chitilde$ induces the identity on $\Gr^W_\bdot$, it follows that $\Theta^\chitilde_{k,k}(\log\sigma_k) = z_k$ for all $k\ge 0$. This implies (\ref{eqn:gamma_0}) when $k=0$. When combined with (\ref{eqn:conj}), it implies that (\ref{eqn:gamma_k}) holds for all $k\ge 1$.

\section{Double Pairings}
\label{sec:pairings}

For potential applications, we consider certain ``double pairings'' and show that their completions are morphisms of MHS. The proof is only outlined as it is similar to that of Theorem~\ref{thm:action}. Double pairings occur in \cite[\S3.2]{kk:intersections}, \cite[\S4.3]{kk:johnson} and have antecedents in \cite{turaev:78} and \cite{massuyeau-turaev}.

\subsection{Topological version}

Suppose that $S$ is a compact oriented surface. Write the boundary of $S$ as a disjoint union
$$
\partial S = \partial' S \amalg                     \partial'' S
$$
of non-empty subsets. Suppose that $x_0',x_1'\in \partial'S$ and that $x_0'',x_1''\in\partial''S$. Note that $\{x_0',x_1'\}$ and $\{x_0'',x_1''\}$ may be singletons. There is a paring
\begin{equation}
\label{eqn:pairing}
\bP_{x_0',x_1'} \otimes \bP_{x_0'',x_1''} \to \bP_{x_0',x_1''} \otimes \bP_{x_0'',x_1'}
\end{equation}
of local systems defined by $\gamma'\gamma'' \otimes \mu'\mu'' \mapsto \gamma'\mu'' \otimes \gamma''\mu'$.

Consider the pairing
$$
H_0(P_{x_0',x_1'} S) \otimes H_0(P_{x_0'',x_1''}S) \to
H_0(P_{x_0',x_1''} S) \otimes H_0(P_{x_0'',x_1'}S)
$$
defined as the composite (dotted arrow) of the solid arrows in the diagram
$$
\xymatrix{
H_0(P_{x_0',x_1'}S) \otimes H_0(P_{x_0'',x_1''}S)
\ar[r] \ar@{.>}[d] &
H_1(S,\partial'S;\bP_{x_0',x_1'})\otimes H_1(S,\partial''S;\bP_{x_0'',x_1''}) \ar[d]^\simeq
\cr
H_0(P_{x_0',x_1''}S) \otimes H_0(P_{x_0'',x_1'}S) & H^1(S,\partial''S;\bP_{x_0',x_1'})\otimes H^1(S,\partial'S;\bP_{x_0'',x_1''}) \ar[d]
\cr
H_0(S;\bP_{x_0',x_1''}) \otimes H_0(S;\bP_{x_0'',x_1'}) \ar[u]^\simeq &
H^2(S,\partial S; \bP_{x_0',x_1'} \otimes \bP_{x_0'',x_1''}) \ar[d]
\cr
H_0(S;\bP_{x_0',x_1''} \otimes \bP_{x_0'',x_1'}) \ar[u] & H^2(S,\partial S; \bP_{x_0',x_1''} \otimes \bP_{x_0'',x_1'}) \ar[l]
}
$$
Like the Goldman bracket and Kawazumi-Kuno action, it can be defined geometrically. It is continuous in the $I$-adic topology and therefore induces a map
$$
H_0(P_{x_0',x_1'} X)^\wedge \hat{\otimes} H_0(P_{x_0'',x_1''}X)^\wedge \to
H_0(P_{x_0',x_1''} X)^\wedge \hat{\otimes} H_0(P_{x_0'',x_1'}X)^\wedge
$$
on $I$-adic completions.

\subsection{Algebraic version}

Suppose that $\Xbar$ is a compact Riemann surface and that $D$ is a finite subset of $\Xbar$. Set $X=\Xbar-D$. Suppose that $\{p_j',p_j'':j=0,1\}$ is a subset of $D$ satisfying
$$
\{p_0',p_1'\} \cap \{p_0'',p_1''\} = \emptyset.
$$
Note that $p_0'$ may equal $p_1'$ or $p_0''$ may equal $p_1''$. Suppose that, for $j=0,1$, we have non-zero tangent vectors $\vv_j' \in T_{p_j'}\Xbar$ and $\vv_j'' \in T_{p_j''}\Xbar$. We take $\vv_0'=\vv_1'$ when $p_0'=p_1'$, and take $\vv_0''=\vv_1''$ when $p_0''=p_1''$. Then, as above, we have the pairing
\begin{equation}
\label{eqn:bi-pairing}
H_0(P_{\vv_0',\vv_1'} X) \otimes H_0(P_{\vv_0'',\vv_1''}X) \to
H_0(P_{\vv_0',\vv_1''} X) \otimes H_0(P_{\vv_0'',\vv_1'}X).
\end{equation}

\begin{theorem}
The ``bi-pairing'' (\ref{eqn:bi-pairing}) is continuous in the $I$-adic topology and its completion is a morphism of pro-MHS.
\end{theorem}

The proof is similar to that of Theorem~\ref{thm:action}. For simplicity, we assume that $p_0',p_1',p_0'',p_1''$ are distinct. The first step is to show that when $x_0',x_1',x_0'',x_1''$ are distinct, the $I$-adic completion of the natural pairing
$$
H_0(P_{x_0',x_1'} X') \otimes H_0(P_{x_0'',x_1''}X'') \to
H_0(P_{x_0',x_1''} X) \otimes H_0(P_{x_0'',x_1'}X)
$$
is a morphism of MHS, where $X'=X-\{x_0'',x_1''\}$ and $X''=X-\{x_0',x_1'\}$. The proof is similar to the proof of Theorem~\ref{thm:naive} and is relatively straightforward. This pairing defines a pairing of admissible variations of MHS over $X^4 - \Delta$, where $\Delta$ denotes the ``fat diagonal'' in $X^4$. One can take the limit as $x_j'\to p_j'$ along $\vv_j'$ and $x_j'' \to p_j''$ along $\vv_j''$ to obtain a pairing
$$
H_0(P_{\vv_0',\vv_1'} X')^\wedge \otimes H_0(P_{\vv_0'',\vv_1''}X'')^\wedge \to
H_0(P_{\vv_0',\vv_1''} X)^\wedge \otimes H_0(P_{\vv_0'',\vv_1'}X)^\wedge,
$$
where $H_0(P_{\vv_0',\vv_1'} X')^\wedge$ and $H_0(P_{\vv_0'',\vv_1''}X'')^\wedge$ denote the limit MHS.

To complete the proof, one has to prove a splitting result similar to Theorem~\ref{thm:splitting}. Specifically, one has to show that the natural homomorphisms
$$
H_0(P_{\vv_0',\vv_1'} X')^\wedge \to H_0(P_{\vv_0',\vv_1'} X)^\wedge \text{ and } H_0(P_{\vv_0'',\vv_1''}X'')^\wedge \to H_0(P_{\vv_0'',\vv_1''}X)^\wedge
$$
are split in the category of pro-MHS. The proof of this is almost identical to that of Theorem~\ref{thm:splitting}.

\end{document}